\documentclass{amsart}
\usepackage{ifxetex,ifluatex}
\newif\ifxetexorluatex
\ifxetex
  \xetexorluatextrue
\else
  \ifluatex
    \xetexorluatextrue
  \else
    \xetexorluatexfalse
  \fi
\fi

\ifxetexorluatex
  \usepackage{fontspec}
\else
  \usepackage[T1]{fontenc}
  \usepackage[utf8]{inputenc}
  \usepackage{lmodern}
  \fi
\usepackage{float}
\usepackage{graphicx, enumerate, url}
\usepackage{algorithmic}
\usepackage{amssymb, stmaryrd}
\usepackage{amsfonts}
\usepackage{mathrsfs}
\usepackage{hyperref}
\usepackage{cleveref}

\newtheorem{definition}{Definition}
\newtheorem{theorem}{Theorem}
\newtheorem{lemma}{Lemma}
\newtheorem{proposition}{Proposition}
\def\VR{\kern-\arraycolsep\strut\vrule &\kern-\arraycolsep}
\def\vr{\kern-\arraycolsep & \kern-\arraycolsep}

\newtheorem{remark}{Remark}
\numberwithin{equation}{section}

\newcommand{\sfT}{\mathsf{T}}

\newcommand{\ttX}{\mathtt{X}}

\newcommand{\R}{\mathbb{R}}

\newcommand{\rmq}{\mathsf{q}}

\newcommand{\rD}{\mathrm{D}}
\newcommand{\rN}{\mathrm{N}}
\newcommand{\fL}{\mathfrak{L}}

\newcommand{\rgrad}{\mathsf{rgrad}}
\newcommand{\rgradHf}{\mathsf{rgrad}_{\cH, f}}
\newcommand{\rhess}{\mathsf{rhess}}
\newcommand{\rhessHf}{\mathsf{rhess}_{\cH, f}}
\newcommand{\hatfY}{\hat{f}_{\bY}}
\newcommand{\hatfYY}{\hat{f}_{\bY\bY}}
\newcommand{\hgradf}{\mathsf{grad}\hat{f}}
\newcommand{\hhessf}{\mathsf{hess}\hat{f}}

\newcommand{\rK}{\mathrm{K}}

\newcommand{\sym}[1]{\mathrm{sym}_{#1}}
\newcommand{\fsym}{\mathrm{symf}}

\newcommand{\asym}[1]{\mathrm{skew}_{#1}}
\newcommand{\KK}{\mathbb{K}}

\newcommand{\C}{\mathbb{C}}

\newcommand{\cX}{\mathcal{X}}

\newcommand{\cB}{\mathcal{B}}

\newcommand{\cD}{\mathcal{D}}
\newcommand{\rcD}{\mathring{\mathcal{D}}}
\newcommand{\cE}{\mathcal{E}}

\newcommand{\Herm}[3]{\mathrm{Sym}_{#1, #2, #3}}
\newcommand{\AHerm}[3]{\mathrm{Skew}_{#1, #2, #3}}
\newcommand{\St}[3]{\mathrm{St}_{#1, #2, #3}}

\newcommand{\Sp}[3]{\mathrm{S}^{+}_{#1, #2, #3}}
\newcommand{\Sd}[2]{\mathrm{S}^{+}_{#1, #2}}
\newcommand{\UU}[2]{\mathrm{U}_{#1, #2}}
\newcommand{\cH}{\mathcal{H}}
\newcommand{\xiH}{\xi_{\cH}}
\newcommand{\etaH}{\eta_{\cH}}

\newcommand{\ft}{\mathfrak{t}}

\newcommand{\hatf}{\hat{f}}

\newcommand{\cM}{\mathcal{M}}
\newcommand{\cU}{\mathcal{U}}

\newcommand{\bI}{I}

\newcommand{\bC}{C}

\newcommand{\ba}{a}

\newcommand{\be}{e}
\newcommand{\bx}{x}

\newcommand{\bdh}{\boldsymbol{\hat{d}}}

\newcommand{\bY}{Y}

\newcommand{\sfg}{\mathsf{g}}

\DeclareMathOperator{\diag}{diag}
\DeclareMathOperator{\Tr}{Tr}
\DeclareMathOperator{\xtrace}{xtrace}

\DeclareMathOperator{\TrR}{Tr_{\R}}

\DeclareMathOperator{\Real}{Re}

\DeclareMathOperator{\Null}{Null}

\DeclareMathOperator{\JJ}{J}

\DeclareMathOperator{\Flag}{Flag}
\begin{document}
\title[Riemannian Hessian]{Operator-valued formulas for Riemannian Gradient and Hessian and families of tractable metrics}
\author{Du Nguyen}
\email{nguyendu@post.harvard.edu}
\begin{abstract}We provide an explicit formula for the Levi-Civita connection and Riemannian Hessian for a Riemannian manifold that is a quotient of a manifold embedded in an inner product space with a non-constant metric function. Together with a classical formula for projection, this allows us to evaluate Riemannian gradient and Hessian for several families of metrics on classical manifolds, including a family of metrics on Stiefel manifolds connecting both the constant and canonical ambient metrics with closed-form geodesics. Using these formulas, we derive Riemannian optimization frameworks on quotients of Stiefel manifolds, including flag manifolds, and a new family of complete quotient metrics on the manifold of positive-semidefinite matrices of fixed rank, considered as a quotient of a product of Stiefel and positive-definite matrix manifold with affine-invariant metrics. The method is procedural, and in many instances, the Riemannian gradient and Hessian formulas could be derived by symbolic calculus. The method extends the list of potential metrics that could be used in manifold optimization and machine learning.
\end{abstract}
\maketitle
\section{Introduction}
In this article, we attempt to address the problem: given a manifold, described by constraint equations and a metric, also defined by an analytic formula, compute the Levi-Civita connection, Riemannian Hessian, and gradient for a function on the manifold. By computing, we mean a procedural, not necessarily closed-form approach. We are looking for a sequence of equations, operators, and expressions to solve and evaluate, rather than starting from a distance minimizing problem. We believe that the approach we take, using a classical formula for projections together with an adaptation of the Christoffel symbol calculation to ambient space addresses the problem effectively for many manifolds encountered in applications. This method provides a very explicit and transparent procedure that we hope will be helpful to researchers in the field. The main feature is it can handle manifolds with non-constant embedded metrics, such as Stiefel manifolds with the canonical metrics, or the manifold of positive-definite matrices with the affine-invariant metrics. The method allows us to compute Riemannian gradients and Hessian for several new families of metrics on manifolds often encountered in applications, including optimization and machine learning. While the effect of changing metrics on first-order optimization methods has been considered previously, we hope this will lead to future works on adapting metrics to second-order methods. The approach is also suitable in the case where the gradient formula is not of closed-form. We also provide a number of useful identities known in special cases.

As an application of the method developed here, we give a short derivation of the gradient and Hessian of a family of metrics studied recently in \cite{ExtCurveStiefel}, extending both the canonical and embedded metric on the {\it Stiefel manifolds with a closed-form geodesics formula}. We derive the Riemannian framework for the induced metrics on quotients of Stiefel manifolds, including {\it flag manifolds}. We also give {\it complete metrics on the fixed-rank positive-semidefinite matrix manifolds}, with efficiently computable gradient, Hessian and geodesics.
\subsection{Background}
In the foundational paper \cite{Edelman_1999}, the authors computed the geodesic equation for a Stiefel manifold $\St{\R}{p}{n}$ of matrices $Y \in \R^{n\times p}$ satisfying $Y^{\mathsf{T}}Y = I_p$, with both the constant metric $\Tr(\eta^{\mathsf{T}}\eta)$ on a tangent vector $\eta$, and the {\it canonical metric} $\Tr(\eta^{\mathsf{T}}(I_p-\frac{1}{2}YY^{\mathsf{T}})\eta)$ using calculus of variation. "Doing so is tedious" (\cite{Edelman_1999}), so the details of the calculations were not included in the paper (\cite{ExtCurveStiefel} recently provided a full derivation). Many examples in the literature usually start with a manifold with a known geodesic equation and construct new manifolds from there. In contrast, we will prove several general formulas for Riemannian gradient, Hessian, and geodesic equations applicable when a subspace of the {\it tangent space} of a manifold is identified as a subspace of a fixed inner product space. This subspace of the tangent space description often arises directly from the description of the manifold by constraint or by quotient formulation. For example, in the Stiefel case, the tangent space at a matrix $Y\in \St{\R}{p}{n}$ is identified with the nullspace of the operator $\JJ(Y):\eta\mapsto Y^{\mathsf{T}}\eta +\eta^{\mathsf{T}}Y$, ($\eta\in\R^{n\times p}$) which follows from the defining equation of the Stiefel manifold. The Grassmann manifold, the quotient of a Stiefel manifold by the right multiplication action of the orthogonal group, has its tangent space identified with the subspace $Y^{\mathsf{T}}\eta=0$ \cite{Edelman_1999}. Given this operator $\JJ(Y)$ and an algebraic formula for the metric, our procedure can produce both the Riemannian gradient and Hessian bypassing calculus of variation, with a relatively short calculation for commonly encountered manifolds. The formulas also suggest a procedural approach to compute the gradient or Hessian when they cannot be reduced to simple expressions. In this approach, the Christoffel symbols are replaced by an operator-valued function called the Christoffel function $\Gamma$ (called $\Gamma_c$ in \cite{Edelman_1999}). In subsequent work, we show the Christoffel function could be used to compute Riemannian curvatures. Thus, this approach has further potentials both in theory and in practice.

Symbolic algebra and differentiation have played a major role in optimization problems arising from machine learning applications in recent years. For manifold applications, it has also been used in \cite{JMLR:v17:16-177,Geomstats}. A simple adoption of symbolic algebra for noncommutative variables is helpful in deriving formulas for projections and connections. While none of our results are dependent on symbolic calculus, it is helpful for sanity check and exploration. The seemingly complicated Christoffel symbols can be handled symbolically due to two facts: first, directional derivatives of matrix expressions can be evaluated rule-based; second, for the trace inner product, index raising is also a symbolic manipulation (e.g., gradients of $\Tr(AxB)$ and $\Tr(Ax^{\mathsf{T}}B)$ with respect to $x$ are simple algebraic expressions). For relatively complex examples with non-constant ambient metrics, we found it could produce the correct gradients and Hessian mostly automatically when applying our procedure. To keep focus, we will not discuss this topic in further detail and only mention that we derive Riemannian frameworks for several manifolds symbolically in several notebooks in our code repository \cite{Nguyen2020riemann}.

Computing the Riemannian connection and the geodesic equation are important steps in understanding the geometry of a manifold in applied problems. We hope this paper provides a step in making this computation more accessible.

\subsection{Riemannian gradient, Hessian and Levi-Civita connection}
First-order approximation of a function $f$ on $\R^n$ relies on the computation of the gradient and the second-order approximation relies on the Hessian matrix or Hessian-vector product. When a function $f$ is defined on a Riemannian manifold $\cM$, the relevant quantities are the Riemannian gradient, a vector field providing first-order approximation for $f$ on the manifold, and Riemannian Hessian, which provides the second-order term.

When a manifold $\cM$ is embedded in an inner product space $\cE$ with the inner product denoted by $\langle,\rangle_{\cE}$, if we have a function $\sfg$ from $\cM$ with values in the set of positive-definite operators operating on $\cE$, we can define an inner product of two vectors $\omega_1, \omega_2\in \cE$ by $\langle \omega_1, \sfg(Y)\omega_2\rangle_{\cE}$ for $Y\in\cM$. This induces an inner product on each tangent space $T_Y\cM$ and hence a Riemannian metric on $\cM$, assuming sufficient smoothness. In this setup, the Riemannian gradient can be computed via a projection from $\cE$ to $T_Y\cM$ (we have the same picture for a horizontal subspace of $T_Y\cM$). In the theory of generalized least squares (GLS) \cite{aitken1936}, it is well-known  that a projection to the nullspace of a full-rank matrix $J$ in an inner product space equipped with a metric $g$ (also represented by a matrix) is given by the formula $I_{\cE}-g^{-1}J^{\mathsf{T}}(Jg^{-1}J^{\mathsf{T}})^{-1}J$ ($I_{\cE}$ is the identity matrix\slash operator of $\cE$). If the tangent space is the nullspace of an operator $\JJ$, and an operator $\sfg$ is used to describe the metric instead of matrices $J$ and $g$, we have a similar formula where the transposed matrix $J^{\mathsf{T}}$ is replaced by the adjoint operator $\JJ^{\ft}$. As mentioned, $\JJ$ is often available explicitly. This projection formula is not often used in the literature, the projection is usually derived directly by minimizing the distance to the tangent space. It turns out when $\JJ$ is given by a matrix equation, $\JJ^{\ft}$ is simple to compute. For manifolds common in applications, $\JJ\sfg^{-1}\JJ^{\ft}$ could often be inverted efficiently, as we will see in several examples. Thus, this will be our main approach to compute the Riemannian gradient.

The Levi-Civita connection of the manifold, which allows us to take covariant derivatives of the gradient, is used to compute the Riemannian Hessian. A vector field $\xi$ in our context could be considered as a $\cE$-valued function from $\cM$, such that $\xi(Y)\in T_Y\cM$ for all $Y\in \cM$. For two vector fields $\xi, \eta$ on $\cM$, the directional derivative $\rD_{\xi}\eta$ is an $\cE$-valued function but generally not a vector field (i.e. not a  $T\cM$-valued function). A covariant derivative \cite{ONeil1983} (or connection) associates a vector field $\nabla_{\xi}\eta$ to two vector fields $\xi, \eta$ on $\cM$. The association is linear in $\xi$, additive in $\eta$ and satisfies the product rule
$$\nabla_{\xi}(f\eta) = f\nabla_{\xi}\eta + (\rD_{\xi}f)\eta$$
for a function $f$ on $\cM$, where $\rD_{\xi}f$ denotes the Lie derivative of $f$ (the directional derivative of $f$ along direction $\xi_x$ at each $x\in\cM$). For a Riemannian metric $\langle,\rangle_R$ on $\cM$, the Levi-Civita connection is the unique connection that is compatible with metric, $\rD_{\xi}\langle \eta, \phi\rangle_R = \langle\nabla_{\xi} \eta,\phi\rangle_R +\langle\eta, \nabla_{\xi}\phi\rangle_R$ ($\phi$ is another vector field), and torsion-free, $\nabla_{\xi}\eta - \nabla_{\eta}\xi = [\xi, \eta]$. If a coordinate chart of $\cM$ is identified with an open subset of $\R^n$ and $\langle,\rangle_R$ is given by a positive-definite operator $\sfg_R$, (i.e. $\langle\xi, \eta\rangle_R = \langle\xi,\sfg_R \eta\rangle_{\R^n}$)
$$\nabla_{\xi}\eta = \rD_{\xi}\eta + \frac{1}{2}\sfg_R^{-1}((\rD_{\xi}\sfg_R)\eta + (\rD_{\eta}\sfg_R)\xi -\cX(\xi, \eta))$$
where $\cX(\xi, \eta)\in\R^n$ (uniquely defined) satisfies $\langle\rD_{\phi}\sfg_R\xi, \eta\rangle_{\R^n} = \langle \phi, \cX(\xi, \eta)\rangle_{\R^n}$ for all vector field $\phi$. The formula is valid for each coordinate chart, and it is often given in terms of Christoffel symbols in index notation (\cite{ONeil1983}, proposition 3.13). We will generalize this operator formula. The gradient and $\cX$ are examples of index-raising, translating from a (multi)linear scalar function $f$ to a (multi)linear vector-valued function of one less variable, that evaluates back to $f$ using the inner product pairing.

The Riemannian Hessian could be provided in two formats, as a bilinear form $\rhess_f^{02}(\xi, \eta)$, returning a scalar function to every two vector fields $\xi, \eta$ on the manifold, or a (Riemannian) Hessian vector product $\rhess^{11}_f\xi$, an operator returning a vector field given a vector field input $\xi$. In optimization, as we need to invert the Hessian in second-order methods, the Riemannian Hessian vector product form $\rhess_f^{11}$ is more practical. However, $\rhess^{02}_f$ is directly related to the Levi-Civita connection (see \cref{eq:rhess02} below), and can be read from the geodesic equation: In \cite{Edelman_1999}, the authors showed the geodesic equation (for a Stiefel manifold) is given by $\ddot{Y} + \Gamma(\dot{Y}, \dot{Y})=0$ where the Christoffel function $\Gamma$ (defined below) maps two vector fields to an ambient function and the bilinear form $\rhess^{02}_f$ is $\hatfYY(\xi, \eta) - \langle\Gamma(\xi, \eta), \hatfY\rangle_{\cE}$. Here, $\hatfY$ and $\hatfYY$ are the ambient gradient and Hessian defined in \cref{sec:ambient}. We will provide an explicit formula for $\Gamma$.

\subsection{Riemannian optimization}
The reader can consult \cite{Edelman_1999,AMS_book} for details of Riemannian optimization, including the basic algorithms once the Euclidean and Riemannian gradient and Hessian are computed. In essence, it has been recognized that many popular equation solving and optimization algorithms on Euclidean spaces can be extended to a manifold framework (\cite{Gabay1982,Edelman_1999}). Many first-order algorithms (steepest descent, conjugate gradient) on real vector spaces could be extended to manifolds using the Riemannian gradient defined above together with a {\it retraction} (some algorithm, for example, conjugate-gradient, requires more differential-geometric measures). Also, using the Riemannian Hessian, second-order optimization methods, for example, Trust-Region (\cite{AMS_book}), could be extended to manifold context. At the $i$-th iteration step, an optimization algorithm produces a tangent vector $\eta_{i}$ to the manifold point $Y_i$, which will produce the next iteration point $Y_{i+1}$ via a retraction (\cite{AdlerShub}, chapter 4 of \cite{AMS_book}). For manifolds considered in this article, effective retractions are available.
  \subsection{Notations}\label{sec:notation} We will attempt to state and prove statements for both the real and Hermitian cases at the same time when there are parallel results, as discussed in \cref{subsec:matrix}. The base field $\KK$ will be $\R$ or $\C$. We use the notation $\KK^{n\times m}$ to denote the space of matrices of size $n\times m$ over $\KK$. We consider both real and complex vector spaces as real vector spaces, and by $\TrR$ we denote the trace of a matrix in the real case or the real part of the trace in the complex case. A real matrix space is a real inner product space with the Frobenius inner product $\TrR ab^{\mathsf{T}}$, while a complex matrix space becomes a real inner product space with inner product $\TrR ab^{\mathsf{H}}$ (see \cref{subsec:matrix} for details). We will use the notation $\ft$ to denote the real adjoint $\mathsf{T}$ for a real vector space, and Hermitian adjoint $\mathsf{H}$ for a complex vector space, both for matrices and operators. We denote $\sym{\ft}X = \frac{1}{2}(X + X^{\ft})$, $\asym{\ft}X = \frac{1}{2}(X - X^{\ft})$. We denote by $\Herm{\ft}{\KK}{n}$ the space of $\ft$-symmetric matrices $X\in \KK^{n\times n}$ with $X^{\ft} = X$. The $\ft$-antisymmetric space $\AHerm{\ft}{\KK}{n}$ is defined similarly. Symbols $\xi, \eta$ are often used to denote tangent vector or vector fields, while $\omega$ is used to denote a vector on the ambient space. The directional derivative in direction $v$ is denoted by $\rD_{v}$, it applies to scalar, vector, or operator-valued functions. If $\ttX$ is a vector field and $f$ is a function, the Lie derivatives will be written as $\rD_{\ttX}f$. We also apply Lie derivatives on scalar or operator-valued functions when $\ttX$ is a vector field, and write $\rD_{\ttX}\sfg$ for example, where $\sfg$ is a metric operator. Because the vector field $\ttX$ may be a matrix, we prefer the notation $\rD_{\ttX}\sfg$ to the usual Lie derivative notation $\ttX\sfg$ which may be ambiguous. By $\UU{\KK}{d}$ we denote the group of $\KK^{d\times d}$ matrices $U$ satisfying $U^{\ft}U = I_d$ (called $\ft$-orthogonal), thus $\UU{\KK}{d}$ is the real orthogonal group $\mathrm{O}(d)$ when $\KK=\R$ and unitary group $\mathrm{U}(d)$ when $\KK=\C$.

  In our approach, a subspace $\cH_Y$ of the tangent space at a point $Y$ on a manifold $\cM$ is defined as either the nullspace of an operator $\JJ(Y)$, or the range of an operator $\rN(Y)$, both are operator-valued functions on $\cM$. Since we most often work with one manifold point $Y$ at a time, we sometimes drop the symbol $Y$ to make the expressions less lengthy. Other operator-valued functions defined in this paper include the ambient metric $\sfg$, the projection $\Pi_{\cH,\sfg}$ to $\cH$, the Christoffel metric term $\rK$, and their directional derivatives. We also use the symbols $\hatfY$ and $\hatfYY$ to denote the ambient gradient and Hessian ($Y$ is the manifold variable). We summarize below symbols and concepts related to the main ideas in the paper, with the Stiefel case as an example (details explained in sections below).
  \begin{center}
    \begin{tabular}{ c l}
      \hline
      Symbol & Concept \\
      \hline
      $\cE$  & ambient space, $\cM$ is embedded in $\cE$. (e.g. $\KK^{n\times p}$)\\
      $\cE_{\JJ}, \cE_{\rN}$  & inner product spaces, range of $\JJ(Y)$ and domain of $\rN(Y)$ below.\\
      $\cH$ & A subbundle of $T\cM$. Either $T\cM$ or a horizontal bundle in practice.\\

 $\JJ(Y)$ & operator from $\cE$ {\it onto} $\cE_{\JJ}$, $\Null(\JJ(Y))=\cH_Y\subset T_Y\cM$. (e.g. $Y^{\ft}\omega + \omega^{\ft}Y$) \\ 
      $\rN(Y)$ & inject. oper. from $\cE_{\rN}$ to $\cE$ onto $\cH_Y\subset T_Y\cM$ (e.g. $\rN(A, B) = Y A + Y_{\perp} B$) \\
      $\xtrace$ & index-raising operator for the trace (Frobenius) inner product\\
      $\sfg(Y)$ & metric given as self-adjoint operator on $\cE$ (e.g. ($\alpha_1YY^{\ft} + \alpha_0 Y_{\perp}Y_{\perp}^{\ft})\eta$) \\
      $\Pi_{\cH,\sfg}$ & projection to $\cH\subset T\cM$ in \cref{prop:NJ}.\\
      $\rK(\xi, \eta)$ & Chrisfl. metric term  $\frac{1}{2}((\rD_{\xi}\sfg)\eta + (\rD_{\eta})\sfg\xi-\xtrace(\langle(\rD_\phi\sfg)\xi, \eta\rangle_{\cE}, \phi))$ \\
      $\Gamma_{\cH}(\xi, \eta)$ & Chrisfl. function $\Pi_{\cH,\sfg}\sfg^{-1}\rK(\xi, \eta)-(\rD_{\xi}\Pi_{\cH, \sfg})\eta$
\end{tabular}
  \end{center}
\subsection{Our contribution}\label{sec:contribution}
We formally define a concept of ambient space with a metric operator to make precise the conditions for our approach. Under these conditions, given the operator-valued function $\JJ$ defining a subbundle $\cH$ of the tangent bundle (in practice, $\cH$ is the full tangent bundle for an embedded manifold or the horizontal subbundle corresponding to a quotient manifold) and an operator-valued function $\sfg$ describing the metric, if $\hatfY$ is the ambient (Euclidean) gradient, we can use the formula $\sfg^{-1}\hatfY-\sfg^{-1}\JJ^{\ft}(\JJ\sfg\JJ^{\ft})^{-1}\JJ\sfg^{-1}\hatfY$ to evaluate the Riemannian gradient. We also provide explicit formulas for the Levi-Civita connection \cref{eq:useGamma} and Riemannian Hessian \cref{eq:rhess11}. Among the results, in case $\cH= T\cM$, identifying the tangent space of a manifold $\cM$ with a subspace of an ambient space $\cE$, the Riemannian Hessian product of a function $f$ for a vector field $\xi$ is given by:
$$\rhess_f^{11}\xi =\Pi_{\sfg}\sfg^{-1}(\hatfYY\xi + \sfg(\rD_{\xi}\Pi_{\sfg})(\sfg^{-1}\hatfY)
-(\rD_{\xi}\sfg)(\sfg^{-1}\hatfY)+\rK(\xi, \Pi_{\sfg}\sfg^{-1}\hatfY))$$
Here, $\Pi_{\sfg}$ is the projection from the ambient space to the tangent space, $\hatfY$ and $\hatfYY$ are the ambient gradient and Hessian (will be defined precisely below) and $\rK$ is the Christoffel metric term. The identification of the tangent space of a manifold $\cM$ with a subspace of $\cE$ makes $\Pi_{\sfg}$ and $\sfg$ operator-valued functions on $\cM$, thus their directional derivatives are well-defined. So $\rD_{\xi}\sfg, \rD_{\xi}\Pi_{\sfg}$ and $\rK$ are all globally defined operator-valued functions. In general, they can be computed from analytic expressions for the metric and the defining equations for the subbundle $\cH$. For the first-order case, we also provide a formula for the gradient when the fiber $\cH_Y$ of $\cH$ at $Y$ is described as the range of a one-to-one map $\rN(Y)$ at each $Y$. The two ways to describe the subspace $\cH_Y$, as the range of $\rN$ or as the nullspace of $\JJ$, gives us two expressions for the gradient (hence the Hessian) that can be evaluated by linear operator algebra.

The difference with the traditional embedded manifold approach (\cite{AMS_book}, Proposition 5.32, or \cite{Mishra2014} for a non-constant metric example), is in those cases, a Riemannian metric {\it for the ambient space} is provided, while we only define the operator-valued metric function for points {\it on the manifold} in our approach. Therefore, we can use a formula such as $\TrR\eta^{\ft}(I-\frac{1}{2}YY^{\ft})\eta$ for a metric on a Stiefel manifold, while that formula is not a metric for an arbitrary matrix $Y$. To use the traditional approach for a nonconstant metric, we need to extend the metric to an ambient manifold, then project the ambient connection back to the embedded manifold. There are many ways to extend the metric, but they all produce the same connection on the embedded manifold, thus, it is plausible that there is a connection formula involving only the metric operator. We provide this formula. The approach could be extended to quotients of embedded manifolds easily. Differentiating the metric operator is straightforward for classical manifolds, while extending the metric by an analytic formula to a large ambient space shifts the problem to evaluating the ambient connection, which is may also be complicated (especially if we want to use a metric formula that works on the full ambient space. In contrast, for a Stiefel manifold, our approach essentially extends the formula $\TrR\eta^{\ft}(I-\frac{1}{2}YY^{\ft})\eta$ to an ambient neighborhood of the manifold but not the full ambient space). We demonstrate this approach in several examples:
\begin{itemize}
\item[1.] For a Stiefel manifold $\St{\KK}{d}{n}$ and two positive numbers $\alpha_0, \alpha_1$, we consider a family of metrics of the form $\TrR(\alpha_0\eta^{\ft}\eta +(\alpha_1-\alpha_0)\eta^{\ft}YY^{\ft}\eta)$ for a tangent vector  $\eta$ at $Y\in \St{\KK}{d}{n}$, generalizing both the constant ambient metric (with $\alpha_0=\alpha_1=1$) and the canonical metric ($\alpha_0=2\alpha_1 = 1$) defined in \cite{Edelman_1999}. These metrics have been studied recently in \cite{ExtCurveStiefel}, with the geodesic equation derived by calculus of variation. We show its gradient and Hessian could also be derived easily in our framework. In \cite{NguyenGeodesic}, we give computationally efficient geodesic formulas generalizing those in \cite{Edelman_1999}, and propose an algorithm for the geodesic distance of these metrics.
\item[2.]  We show it is simple to obtain a Riemannian optimization framework for quotients of a Stiefel manifold $\St{\KK}{d}{n}$ by a group of isometries. We consider the case where this group is a block-diagonal group. This case includes the flag manifolds \cite{Nishi,YeLim}. We treat the full family of metrics in \cite{ExtCurveStiefel}. The formulas obtained have the same form as those of Stiefel manifolds, and we can provide the Riemannian Hessian operator and second-order methods, an improvement over previous works. In \cite{NguyenGeodesic} we also propose an algorithm for geodesic distance on flag manifolds.
\item[3.] We apply our formula to the manifold $\Sd{\KK}{n}$ of positive-definite matrices in $\KK^{n\times n}$ with the affine-invariant metrics. We recover the known expressions for Riemannian gradient and Hessian (see \cref{sec:related} for related works).
\item[4.] We consider the quotient metric for the manifold of positive-semidefinite matrices of rank $d$ in $\KK^{n\times n}$, $\Sp{\KK}{d}{n}$, considered as $(\St{\KK}{d}{n}\times \Sd{\KK}{d})/\UU{\KK}{d}$, the quotient by an orthogonal group of the product of a Stiefel manifold equipped with a metric in the family introduced above and the positive-definite matrix manifold with the affine-invariant metrics. We show the Riemannian gradient and Hessian could be computed with the help of a Lyapunov-type equation \cref{eq:eq_lyapunov} (solvable at cost $O(d^3)$ adapting the Bartels-Stewart algorithm). The affine-invariant behavior of the $\Sd{\KK}{d}$ part distinguishes this metric from others considered in the literature. While it shares some features with a metric considered in \cite{BonSel}, this metric comes from a Riemannian submersion, allowing us to compute the Hessian by our framework. With this metric, $\Sp{\KK}{d}{n}$ is a {\it complete} manifold with known geodesics. It could be shown (not presented here) an approximate geodesic considered in \cite{BonSel} could be considered as a limiting geodesic, where our Riemannian metric becomes degenerate.
\item[5.] We implement the operators in our approach in {\it Python} that can be used with a manifold optimization package, extending the base class Manifold in the package {\it Pymanopt} (based on {\it Manopt}, \cite{JMLR:v17:16-177}, \cite{manopt}), see \cite{Nguyen2020riemann}. Note that our implementation can take metric parameters, while most implementations typically handle one metric per module. With our approach, it is simple to test the torsion-free and metric-compatibility properties of the Levi-Civita connection numerically as the derivatives involved are of operator-valued functions. We tested the implementation with the Trust-Region algorithm for a few benchmark problems.
\end{itemize}
Our analysis offers two theoretical insights in deciding potential metric candidates in optimization problems:
\begin{itemize}
\item[1.] Non-constant ambient metrics may have the same big-$O$ time picture as the constant one. This is the case with the examples above when the constraint and the metrics are given in matrix polynomials or inversion. If the ambient Hessian could be computed efficiently, in many cases the Riemannian Hessian expressions (maybe tedious algebraically), could be computed by operator composition with the same order-of-magnitude time complexity as the Riemannian gradient. This suggests non-constant metrics may be competitive if the improvement in convergence rate is significant. For certain problems involving positive-definite matrices, a non-constant metric is a better option(\cite{Sra}).
\item[2.] There is a theoretical bound for the cost of computing the gradient, assuming that the metric $\sfg$ is easy to invert. If the complexity of computing $\sfg$ and $\JJ$ is known, it remains to estimate the cost of inverting $\JJ\sfg^{-1}\JJ^{\ft}$ (or $\rN^{\ft}\sfg\rN$). While in our examples these operators are reduced to simple ones that could be inverted efficiently, otherwise, $\JJ\sfg^{-1}\JJ^{\ft}$ could be solved by a conjugate gradient (CG)-method (which has been used for Riemannian optimization, see \cite{Kasai16}). In that case, the time cost is proportional to the rank of $\JJ$ (or $\rN))$ times the cost of each CG step, which can be estimated depending on the problem.
\end{itemize}
\subsection{Related works and outline}\label{sec:related}
The formulas for the Hessian (\cref{prop:covar}) have mostly been used with a constant metric on an ambient manifold. For example, formula (7) of \cite{AMT} is the special case for our Hessian formula for constant ambient metrics, or section 4.9 of \cite{Edelman_1999} also discusses the Christoffel function for an embedded manifold. As discussed, we provide a formalism allowing us to compute this function via operators on ambient space. The paper \cite{Manton} gave the original treatment of the Hessian for the unitary/complex case. The formulas for Stiefel manifolds overlap with those in \cite{Edelman_1999,ExtCurveStiefel}, obtained by different methods. Optimization on flag manifolds was studied in \cite{Nishi,YeLim}. The affine-invariant metric on positive-definite matrices was also widely studied, for example in \cite{Smith_2005,Pennec,Sra,JeurisVal}. There are numerous metrics on the fixed-rank positive-semidefinite (PSD) manifolds, we mentioned \cite{BonSel} that motivated our approach. Although working with the same product of Stiefel and positive-definite manifolds with the affine-invariant metric, that paper did not use the Riemannian submersion metric on the quotient and focused on first-order methods. We compute the Levi-Civita connection for second-order methods. In \cite{VAV,JBAS}, two different families of metrics on PSD manifolds are studied. They both require solving Lyapunov equations but have different behaviors on the positive-definite part. Articles \cite{Manton,SepulchreMishra} discuss the effect of adapting metrics to optimization problems (\cite{SepulchreMishra} adapts ambient metrics to the objective function using first-order methods.)

In the next section, we formulate and prove the main theoretical results of the paper. We then identify the adjoints of common operators on matrix spaces. We apply the theory developed to the manifolds discussed above. We then discuss numerical results and implementation. We conclude with a discussion of future directions.

\section{Ambient space and optimization on Riemannian manifolds}\label{sec:ambient}
While an abstract manifold is defined in terms of coordinate charts, manifolds appearing in applications are usually subspaces of an inner product space $\cE$ with inner product $\langle,\rangle_{\cE}$, or quotient space of such subspaces, defined by constraints and symmetries (\cite{Edelman_1999,AMS_book}). For a manifold $\cM$ embedded in $\cE$ differentiably, its tangent spaces are also identified with subspaces of $\cE$, the identification is parametrized smoothly by points of $\cM$. For example, if $\cM$ is a unit sphere in $\cE=\R^m$, the tangent space at a point $\bx$ is the subspace of vectors $\eta$ such that $\bx^{\mathsf{T}}\eta=0$. In the quotient manifold situation (reviewed below), for $Y\in\cM$, we focus on a subspace $\cH_Y$ of $T_Y\cM$, the horizontal subspace with respect to a group action. The space $\cH_Y$ is also considered as a subspace of $\cE$, allowing us to treat quotient manifolds in the same setting.

If $\hat{f}$ is a scalar function from an open subset $\cU$ of $\cE$, its gradient $\hgradf$ satisfies $\langle\hat{\eta}, \hgradf\rangle_{\cE} = \rD_{\hat{\eta}} \hat{f}$ for all vector fields $\hat{\eta}$ on $\cU$ where $\rD_{\hat{\eta}} \hatf$ is the Lie derivative of $\hatf$ with respect to $\hat{\eta}$. As well-known \cite{Edelman_1999,AMS_book}, the Riemannian gradient and Hessian product of a function $f$ on $\cM$ could be computed from the Euclidean gradient and Hessian, which are evaluated by extending $f$ to a function $\hat{f}$ on a region of $\cE$ near $\cM$. The process is independent of the extension $\hat{f}$.

\begin{definition}\label{def:ambient} We call an inner product (Euclidean) space $(\cE, \langle,\rangle_{\cE})$  an embedded ambient space of a Riemannian manifold $\cM$ if there is a differentiable (not necessarily Riemannian) embedding $\cM\subset\cE$.

  Let $f$ be a function on $\cM$ and $\hat{f}$ be an extension of $f$ to an open neighborhood of $\cE$ containing $\cM$. We call $\hgradf$ an ambient gradient of $f$. It is a vector-valued function from $\cM$ to $\cE$ such that for all vector fields $\eta$ on $\cM$
\begin{equation}
    \label{eq:amb}
     \langle\eta(Y), \hgradf(Y)\rangle_{\cE} = (\rD_{\eta(Y)} f)(Y)\text{ for all }Y\in \cM
\end{equation}
or equivalently $\langle\eta, \hgradf\rangle_{\cE} = \rD_\eta f$. Given an ambient gradient $\hgradf$ with continuous derivatives, we define the ambient Hessian to be the map $\hhessf$ associating to a vector field $\xi$ on $\cM$ the derivative $\rD_{\xi}\hgradf$. We define the ambient Hessian bilinear form $\hhessf^{02}(\xi, \eta)$ to be $\langle(\rD_{\xi}\hgradf),\eta\rangle_{\cE}$. If $Y\in\cM$ is considered as a variable, we also use the notation $\hatfY$ for $\hgradf$ and $\hatfYY$ for $\hhessf$.
\end{definition}
By the Whitney embedding theorem, any manifold has an ambient space. Coordinate charts could be considered as a collection of compatible local ambient spaces. However, globally defined ambient spaces are more effective in computations.

From the embedding $\cM\subset\cE$, the tangent space of $\cM$ at each point $Y\in\cM$ is considered as a subspace of $\cE$. Thus, a vector field on $\cM$ could be considered as an $\cE$-valued function on $\cM$ and we can take its directional derivatives. This derivative is dependent on the embedding and hence not intrinsic.
\begin{lemma}
For a function $f$ and two vector fields $\xi, \eta$ on $\cM\subset\cE$ we have:
\begin{equation}
\label{eq:ehess}
\hatfYY^{02}(\xi, \eta)= \hhessf^{02}(\xi, \eta) = \rD_{\xi}(\rD_{\eta} f)- \langle \rD_{\xi}\eta, \hatfY\rangle_{\cE}
\end{equation}
\end{lemma}
\begin{proof} We have $\rD_{\xi}\langle\eta, \hatfY\rangle_{\cE} =
  \langle\rD_{\xi}\eta, \hatfY\rangle_{\cE} + \langle\eta, \rD_{\xi}(\hatfY)\rangle_{\cE} = \rD_{\xi}(\rD_{\eta} f)$ by taking directional derivatives of \cref{eq:amb}, thus $\langle\eta, \rD_{\xi}(\hatfY)\rangle_{\cE}$ can be evaluated by \cref{eq:ehess}.\qed
\end{proof}

We begin with a standard result of inner product spaces. Recall that the adjoint of a linear map $A$ between two inner product spaces $V$ and $W$ is the map $A^{\ft}$ such that $\langle Av, w\rangle_W = \langle v, A^{\ft}w\rangle_V$ where $\langle,\rangle_V, \langle,\rangle_W$ denote the inner products on $V$ and $W$, respectively. If $A$ is represented by a matrix also called $A$ in two orthogonal bases in $V$ and $W$ respectively, then $A^{\ft}$ is represented by its transpose $A^{\mathsf{T}}$. A projection from an inner product space $V$ to a subspace $W$ is a linear map $\Pi_W$ such that $\langle v, w\rangle_V = \langle \Pi_W v, w\rangle_V$ for all $w\in W$, $v\in V$. It is well-known a projection always exists and unique, and $\Pi_W v$ minimizes the distance from $v$ to $W$.
\begin{proposition}\label{prop:JN}
  Let $\cE$ be a vector space with an inner product $\langle,\rangle_{\cE}$. Let $\sfg$ be a self-adjoint positive-definite operator on $\cE$, thus $\langle\sfg\be_1, \be_2\rangle_{\cE}= \langle\be_1, \sfg \be_2\rangle_{\cE}$. The operator $\sfg$ defines a new inner product on $\cE$ by $\langle\be_1, \be_2\rangle_{\cE, \sfg}:= \langle\be_1, \sfg \be_2\rangle_{\cE}$. If $W = \Null(\JJ)$ for a map $\JJ$ from $\cE$ onto an inner product space $\cE_{\JJ}$, the projection $\Pi_{\sfg}=\Pi_{\sfg, W}$ from $\cE$ to $W$  under the inner product $\langle, \rangle_{\cE, \sfg}$ is given by $\Pi_{\sfg}\be = \be - \sfg^{-1}\JJ^{\ft}(\JJ\sfg^{-1}\JJ^{\ft})^{-1}\JJ\be$, where $\JJ^{\ft}$ is the adjoint map of $\JJ$ for all $e\in\cE$.

  Alternatively, if $\rN$ is a one-to-one map from an inner product space $\cE_{\rN}$ to $\cE$ such that $W=\rN(\cE_{\rN})$, then the projection to $W$ could be given by $\Pi_{\sfg}\be = \rN(\rN^{\ft}\sfg\rN)^{-1}\rN^{\ft}\sfg\be$.

  The operators $\sfg\Pi_{\sfg}$ and $\Pi_{\sfg}\sfg^{-1}$ are self-adjoint under $\langle,\rangle_{\cE}$.
\end{proposition}
\begin{proof} The assumption that $\JJ$ is onto $\cE_{\JJ}$ shows $\JJ^{\ft}$ is injective (as $\JJ^{\ft}a = 0$ implies $\langle a, \JJ\omega\rangle_{\cE_{\JJ}} = 0$ for all $\omega\in \cE$, and since $\JJ$ is onto this implies $a = 0$). This in turn implies  $\JJ\sfg^{-1}\JJ^{\ft}$ is invertible as if $\JJ\sfg^{-1}\JJ^{\ft} a = 0$, then $\langle \JJ\sfg^{-1}\JJ^{\ft}a, a\rangle_{\cE} = 0$, so $\langle \sfg^{-1}\JJ^{\ft}a, \JJ^{\ft}a\rangle_{\cE} = 0$ and hence $\JJ^{\ft}a = 0$ as $\sfg$ is positive-definite. We can show $\rN^{\ft}\sfg\rN$ is invertible similarly.

  For the first case, if $e_W \in W=\Null(\JJ)$ and $e\in \cE$,
$$\langle \sfg^{-1}\JJ^{\ft}(\JJ\sfg^{-1}\JJ^{\ft})^{-1}\JJ e, \sfg\be_W\rangle_{\cE} =  \langle\JJ^{\ft}(\JJ\sfg^{-1}\JJ^{\ft})^{-1}\JJ e, \be_W\rangle_{\cE}=\langle(\JJ\sfg^{-1}\JJ^{\ft})^{-1}\JJ\be, \JJ\be_W\rangle_{\cE}$$
  where the last term is zero because $\be_W\in \Null(\JJ)$, so $\langle\Pi_{\sfg} e, \sfg e_W\rangle_{\cE} = \langle e, \sfg e_W\rangle_{\cE}$. For the second case, assuming $\be_W =\rN(\be_{\rN})$ for $\be_{\rN}\in \cE_{\rN}$ then (Using $(AB)^{\ft} = B^{\ft}A^{\ft}$):
$$\langle \rN(\rN^{\ft}\sfg\rN)^{-1}\rN^{\ft}\sfg \be, \sfg\rN(\be_{\rN})\rangle_{\cE} =
\langle\sfg\be, \rN(\rN^{\ft}\sfg\rN)^{-1}\rN^{\ft} \sfg\rN(\be_{\rN})\rangle_{\cE}=\langle\sfg\be, \rN(\be_{\rN})\rangle_{\cE}$$
The last statement follows from the defining equations of $\Pi_{\sfg}$.\qed
\end{proof}
Recall the Riemannian gradient of a function $f$ on a manifold $\cM$ with Riemannian metric $\langle,\rangle_R$ is the vector field $\rgrad_f$ such that $\langle\rgrad_f(Y), \xi(Y)\rangle_R = (\rD_{\xi} f)(Y)$ for any point $Y\in \cM$, and any vector field $\xi$. Let $\cH$ be a subbundle of the tangent bundle $T\cM$, we recall this means $\cH$ is a collection of subspaces (fibers) $\cH_Y\subset T_Y\cM$ for $Y\in \cM$ such that $\cH$ is itself a vector bundle on $\cM$, i.e. $\cH$ is locally a product of a vector space and an open subset of $\cM$, together with a linear coordinate change condition (see \cite{ONeil1983}, definition 7.24 for details). We can define the $\cH$-Riemannian gradient $\rgradHf$ of $f$ as the unique $\cH$-valued vector field such that $\langle\rgradHf, \xi_{\cH}\rangle_R = \rD_{\xi_{\cH}} f$ for any $\cH$-valued vector field $\xi_{\cH}$. Uniqueness follows from nondegeneracy of the inner product restricted to $\cH$. Clearly, when $\cH=T\cM$, $\rgradHf=\rgrad_{T\cM}f$ is the usual Riemannian gradient. We have:
\begin{proposition}\label{prop:NJ}
  Let $(\cE, \langle,\rangle_{\cE})$ be an embedded ambient space of a manifold $\cM$ as in \cref{def:ambient}. Let $\sfg$ be a smooth operator-valued function associating each $\bY\in \cM$ a self-adjoint positive-definite operator $\sfg(Y)$ on $\cE$. Thus, each $\sfg(Y)$ defines an inner product on $\cE$, which induces an inner product on $T_Y\cM$ and hence $\sfg$ induces a Riemannian metric on $\cM$. Let $\cH$ be a subbundle of $T\cM$. Define $\Pi_{\cH,\sfg}$ to be the operator-valued function such that $\Pi_{\cH,\sfg}(\bY)$ is the projection associated with $\sfg(Y)$ from $\cE$ to the fiber $\cH_Y$, and for the case $\cH=T\cM$, define $\Pi_{\cM,\sfg} = \Pi_{T\cM,\sfg}$. For an ambient gradient $\hgradf$ of $f$, the $\cH$-Riemannian gradient of $f$ can be evaluated as:
\begin{equation}
\label{eq:rgrad}
\rgradHf = \Pi_{\cH,\sfg}\sfg^{-1} \hgradf
\end{equation}
If there is an inner product space $\cE_{\JJ}$ and a map $\JJ$ from $\cM$ to the space $\fL(\cE, \cE_{\JJ})$ of linear maps from $\cE$ to $\cE_{\JJ}$, such that for each $\bY\in \cM$, the range of $\JJ(\bY)$ is precisely $\cE_{\JJ}$, and its  nullspace is $\cH_Y$ then $\Pi_{\cH,\sfg}(\bY)e$ for $e\in\cE$ could be given by:
\begin{equation}
\Pi_{\cH,\sfg}(\bY)\be = \be - \sfg^{-1}\JJ^{\ft}(\JJ\sfg^{-1}\JJ^{\ft})^{-1}\JJ\be; \ \text{all are evaluated at }\bY
\end{equation}
If there is an inner product space $\cE_{\rN}$ and a map $\rN$ from $\cM$ to the space $\fL(\cE_{\rN}, \cE)$ of linear maps from $\cE_{\rN}$ to $\cE$ such that for each $\bY\in\cM$, $\rN(\bY)$ is one-to-one, with its range is precisely $\cH_Y$ then:
\begin{equation}
\Pi_{\cH,\sfg}\be = \rN(\rN^{\ft}\sfg\rN)^{-1}\rN^{\ft}\sfg\be; \text{ all are evaluated at }\bY
\end{equation}
\end{proposition}
\begin{proof} For any $\cH$-valued vector field $\xi_{\cH}$, we have:
  $$\langle\Pi_{\cH, \sfg}\sfg^{-1}\hgradf,\sfg \xi_{\cH} \rangle_{\cE} = \langle\hgradf,\Pi_{\cH, \sfg}\sfg^{-1}\sfg \xi_{\cH} \rangle_{\cE} =
  \langle\hatfY, \xi_{\cH} \rangle_{\cE} =\rD_{\xi_{\cH}} f$$
because $\Pi_{\cH, \sfg}\sfg^{-1}$ is self-adjoint and the projection is idempotent. The remaining statements are just a parametrized version of \cref{prop:JN}.\qed
\end{proof}
Note, we are not making any smoothness assumption on $\JJ$ or $\rN$ yet, although $\Pi_{\cH,\sfg}$ is assumed to be sufficiently smooth. In fact, $\rN$ is often not smooth. $\JJ$ is usually smooth as it is constructed from a smooth constraint on $\cM$, or on the horizontal requirements of a vector field.
\begin{definition}A triple $(\cM, \sfg, \cE)$ with $\cE$ an inner product space, $\cM\subset\cE$ a differentiable manifold submersion, and $\sfg$ is a positive-definite operator-valued-function from $\cM$ to $\fL(\cE, \cE)$ is called an embedded ambient structure of $\cM$. $\cM$ is a Riemannian manifold with the metric induced by $\sfg$. 
\end{definition}  
From the definition of Lie brackets, for an embedded ambient space $\cE$ of $\cM$ we have
\begin{equation}
\label{eq:vectorfield}
\rD_{\xi}\eta-\rD_{\eta}\xi = [\xi, \eta]\text{ for all vector fields } \xi, \eta \text{ on } \cM.
\end{equation}
Recall if $\cM, \langle,\rangle_R$ is a Riemannian manifold with the Levi-Civita connection $\nabla$, the Riemannian Hessian (vector product) of a function $f$ is the operator sending a tangent vector $\xi$ to the tangent vector $\rhess_f^{11}\xi = \nabla_{\xi}\rgrad_f$. The Riemannian Hessian bilinear form is the map evaluating on two vector fields $\xi, \eta$ as $\langle\nabla_{\xi}\rgrad_f, \eta\rangle_R$. For a subbundle $\cH$ of $T\cM$ and a $\cH$-valued vector field $\xi_{\cH}$, we define the $\cH$-Riemannian Hessian similarly as $\rhess_{\cH, f}^{11}\xi_{\cH} = \Pi_{\cH, \sfg}\nabla_{\xi_{\cH}}\rgradHf$ and we call $\rhess_{\cH, f}^{02}(\xi_{\cH}, \eta_{\cH}) = \langle\Pi_{\cH, \sfg}\nabla_{\xi_{\cH}}\rgradHf, \eta_{\cH}\rangle_R = \langle\nabla_{\xi_{\cH}}\rgradHf, \eta_{\cH}\rangle_R$ the $\cH$-Riemannian Hessian bilinear form. The next theorem shows how to compute the Riemannian connection and the associated Riemannian Hessian.
\begin{theorem}\label{prop:covar}
Let $(\cM, \sfg, \cE)$ be an embedded ambient structure of a Riemannian manifold $\cM$. There exists an $\cE$-valued bilinear form $\cX$ sending a pair of vector fields $(\xi, \eta)$ to $\cX(\xi, \eta)\in \cE$ such that for any vector field $\xi_0$:
\begin{equation}
  \label{eq:cX}
  \langle\cX(\xi, \eta), \xi_0\rangle_{\cE} = \langle\xi, (\rD_{\xi_0}\sfg) \eta\rangle_{\cE}
\end{equation}
Let $\Pi_{\cM, \sfg}$ be the projection from $\cE$ to the tangent bundle of $\cM$. Then $\Pi_{\cM, \sfg}\sfg^{-1}\cX(\xi, \eta)$ is uniquely defined given $\xi, \eta$ and $\cX(\xi, \eta)$ is also unique if we require $\cX(\xi(Y), \eta(Y))$ to be in $T_Y\cM$ for all $Y\in \cM$. For two vector fields $\xi, \eta$  on $\cM$, define
\begin{equation}
\label{eq:connection}
\begin{gathered}
\rK(\xi, \eta ) := \frac{1}{2}((\rD_{\xi}\sfg)\eta +(\rD_{\eta}\sfg)\xi -\cX(\xi, \eta))\in \cE\\
\hat{\nabla}_{\xi}\eta := \rD_{\xi} \eta + \sfg^{-1}\rK(\xi, \eta)\\
\nabla_{\xi}\eta  := \Pi_{\cM;\sfg}\hat{\nabla}_{\xi}\eta = \Pi_{\cM;\sfg}(\rD_{\xi} \eta + \sfg^{-1}\rK(\xi, \eta))
\end{gathered}
\end{equation}
Then $\nabla_{\xi}\eta$ is the covariant derivative associated with the Levi-Civita connection. It could be written using the Christoffel function $\Gamma$:
\begin{equation}
  \label{eq:useGamma}
  \begin{gathered}
  \Gamma(\xi, \eta) := -(\rD_{\xi}\Pi_{\cM;\sfg})\eta + \Pi_{\cM;\sfg}\sfg^{-1}\rK(\xi, \eta)\\
  \nabla_{\xi}\eta = \rD_{\xi}\eta +\Gamma(\xi, \eta)
  \end{gathered}
\end{equation}
If $\cH$ is a subbundle of $T\cM$, and $\xi_{\cH}, \eta_{\cH}$ are two $\cH$-valued vector fields, we have:
\begin{equation}
  \label{eq:useGammaH}
  \begin{gathered}
    \Pi_{\cH, \sfg}\nabla_{\xiH}\etaH = \rD_{\xiH}\etaH +\Gamma_{\cH}(\xiH, \etaH) \text{ with }\\
  \Gamma_{\cH}(\xiH, \etaH) := -(\rD_{\xiH}\Pi_{\cH;\sfg})\etaH + \Pi_{\cH;\sfg}\sfg^{-1}\rK(\xiH, \etaH)
  \end{gathered}
\end{equation}
If $f$ is a function on $\cM$, $\hatfY$ is an ambient gradient of $f$ and $\hatfYY$ is the ambient Hessian operator, we have $\rhessHf^{11}\xiH := \Pi_{\cH;\sfg}\nabla_{\xiH}\rgradHf$ and $\rhessHf^{02}$ are given by:
\begin{equation}
  \label{eq:rhess11}
  \begin{gathered}
\rhessHf^{11}\xiH = \Pi_{\cH;\sfg}\sfg^{-1}(\hatfYY\xiH +
\sfg(\rD_{\xiH}(\Pi_{\cH,\sfg}\sfg^{-1}))\hatfY +\rK(\xiH, \Pi_{\cH,\sfg}\sfg^{-1}\hatfY))\\
=\Pi_{\cH;\sfg}\sfg^{-1}(\hatfYY\xiH +
\sfg(\rD_{\xiH}\Pi_{\cH, \sfg})\sfg^{-1}\hatfY
-(\rD_{\xiH}\sfg)\sfg^{-1}\hatfY
+\rK(\xiH, \Pi_{\cH, \sfg}\sfg^{-1}\hatfY))
\end{gathered}
\end{equation}
\begin{equation}
\label{eq:rhess02}
\begin{gathered}
\rhessHf^{02}(\etaH, \xiH) = \langle\nabla_{\xiH}\Pi_{\cH,\sfg} \sfg^{-1}\hatfY, \sfg\etaH\rangle_{\cE}= \hatfYY(\xiH, \etaH) -\langle \Gamma_{\cH}(\xiH, \etaH),\hatfY\rangle_{\cE}
\end{gathered}
\end{equation}
\end{theorem}
The form $\Gamma(\xi, \eta)$ appeared in \cite{Edelman_1999} and was computed for the case of a Stiefel manifold, and was called a Christoffel function. It includes the Christoffel metric term $\rK$ and the derivative of $\Pi_{\cM, \sfg}$. Evaluated at $\bY\in\cM$, it depends only on the tangent vectors $\eta(\bY)$ and $\xi(\bY)$, not on the whole vector fields. Equation (2.57) in that reference is the expression of $\rhess^{02}_f$ in terms of $\Gamma$ above. In \cite{Edelman_1999}, $\Gamma_{\cH}$ was computed for a Grassmann manifold. The formulation for subbundles allows us to extend the result to Riemannian submersions and quotient manifolds.
\begin{proof} $\cX$ is the familiar index-raising term: for $Y\in\cM$ and $v_0, v_1, v_2\in T_Y\cM$, as $\langle v_1, (\rD_{v_0}\sfg) v_2\rangle_{\cE}$ is a tri-linear function on $T_Y\cM$ and the Riemannian inner product on $T_Y\cM$ is nondegenerate, the index-raising bilinear form $\tilde{\cX}$ with value in $T_Y\cM$ is uniquely defined, so $\cX(\xi(Y), \eta(Y))=\tilde{\cX}(\xi(Y), \eta(Y))$ satisfies \cref{eq:cX}, where we consider $T_Y\cM$ as a subspace of $\cE$. Thus, we have proved the existence of $\cX$. If we take another $\cE$-valued function $\cX_1$ satisfying the same condition but not necessarily in the tangent space, the expression $\Pi_{\cM,\sfg}\sfg^{-1}\cX_1$, hence $\Pi_{\cM,\sfg}\sfg^{-1}\rK$ is independent of the choice of $\cX_1$, as for three vector fields $\xi_0, \xi, \eta$
  $$\langle\sfg\xi_0, \Pi_{\cM,\sfg}\sfg^{-1}\cX_1(\xi, \eta)\rangle_{\cE} = \langle\sfg\xi_0, \sfg^{-1}\cX_1(\xi, \eta)\rangle_{\cE}=\langle\xi, \rD_{\xi_0}\eta\rangle_{\cE}$$
  We can verify directly that $\nabla_{\xi}\eta$ satisfies the conditions of a covariant derivative: linear in $\xi$ and satisfying the product rule with respect to $\eta$. Similar to the calculation with coordinate charts, we can show $\nabla$ is compatible with metric: for two vector fields $\eta, \xi$, $2\langle\nabla_{\xi}\eta, \sfg\eta \rangle_{\cE}= 2\langle\Pi_{\cM,\sfg}\hat{\nabla}_{\xi} \eta, \sfg\eta \rangle_{\cE}$, which is $2\langle\rD_{\xi} \eta +\sfg^{-1}\rK(\xi, \eta), \sfg\eta \rangle_{\cE}$ by definition and by property of the projection. Expanding the last expression and use $\langle\cX(\xi, \eta), \eta \rangle_{\cE}=\langle(\rD_{\eta}\sfg)\xi,\eta \rangle_{\cE}$
$$\begin{gathered}
2\langle\rD_{\xi} \eta, \sfg\eta \rangle_{\cE}+2\langle\sfg^{-1}\rK(\xi, \eta), \sfg\eta \rangle_{\cE} =\\2\langle\rD_{\xi} \eta, \sfg\eta \rangle_{\cE} +
\langle(\rD_{\xi}\sfg)\eta + (\rD_{\eta}\sfg)\xi -\cX(\xi, \eta), \eta \rangle_{\cE}\\
=2\langle\rD_{\xi}\eta, \sfg\eta \rangle_{\cE} + \langle(\rD_{\xi}\sfg)\eta, \eta \rangle_{\cE} = \rD_{\xi}\langle \eta, \sfg\eta \rangle_{\cE}
  \end{gathered}$$
 Torsion-free follows from the fact that $\rK$ is symmetric and \cref{eq:vectorfield}:
$$\nabla_{\xi}\eta - \nabla_{\eta}\xi = \Pi_{\cM,\sfg}(\rD_{\xi}\eta - \rD_{\eta}\xi) = [\xi, \eta]$$
For \cref{eq:useGamma}, we note $\Pi_{\cM,\sfg}\rD_{\xi}\eta = \rD_{\xi}(\Pi_{\cM,\sfg}\eta) - 
(\rD_{\xi}\Pi_{\cM,\sfg})\eta$ so
$$\nabla_{\xi}\eta = \rD_{\xi}(\eta) - (\rD_{\xi}\Pi_{\cM,\sfg})\eta + \Pi_{\cM, \sfg}\sfg^{-1}\rK(\xi, \eta)$$
For \cref{eq:useGammaH}, for $Y\in\cM$, $\cH_Y\subset T_Y\cM$ implies $\Pi_{\cH, \sfg}\Pi_{\cM, \sfg}= \Pi_{\cH, \sfg}$. Therefore, \cref{eq:connection} implies
$$ \Pi_{\cH, \sfg}\nabla_{\xiH}\etaH  = \Pi_{\cH;\sfg}(\rD_{\xi} \etaH + \sfg^{-1}\rK(\xiH, \etaH)) $$
and as before, we use $\Pi_{\cH,\sfg}\rD_{\xiH}\etaH = \rD_{\xiH}(\Pi_{\cH,\sfg}\etaH) - 
(\rD_{\xiH}\Pi_{\cH,\sfg})\etaH$ and $\Pi_{\cH,\sfg}\etaH=\etaH$. The first line of \cref{eq:rhess11} is by definition and $\rD_{\xi}(\Pi_{\cH,\sfg}\sfg^{-1}\hatfY) = \rD_{\xi}(\Pi_{\cH,\sfg}\sfg^{-1})\hatfY +\Pi_{\cH,\sfg}\sfg^{-1}\hatfYY$. Expand, note $\Pi_{\cH,\sfg}\sfg^{-1}(\sfg 
\Pi_{\cH,\sfg}\rD_{\xiH}\sfg^{-1})\hatfY = -\Pi_{\cH,\sfg}\sfg^{-1}(\rD_{\xiH}\sfg)\sfg^{-1}\hatfY$ (as $\Pi_{\cH,\sfg}$ is idempotent), we have the second line. For \cref{eq:rhess02}:
$$\begin{gathered}
\langle\nabla_{\xiH} \Pi_{\cH,\sfg} (\sfg^{-1}\hatfY), \sfg\etaH\rangle_{\cE}= 
\rD_{\xiH}\langle \Pi_{\cH,\sfg} (\sfg^{-1}\hatfY), \sfg\etaH\rangle_{\cE}-
\langle \Pi_{\cH,\sfg} (\sfg^{-1}\hatfY), \sfg \nabla_{\xi}\etaH\rangle_{\cE}=\\
\rD_{\xiH}\langle\hatfY, \etaH\rangle_{\cE} -\langle\hatfY, \Pi_{\cH,\sfg}\sfg^{-1}\sfg\Pi_{\cH,\sfg}\hat{\nabla}_{\xiH}\etaH\rangle_{\cE}=\rD_{\xiH}\langle\hatfY, \etaH\rangle_{\cE} -\langle\hatfY , \Pi_{\cH,\sfg}\hat{\nabla}_{\xiH}\etaH\rangle_{\cE}\\
=\rD_{\xiH}\langle\hatfY, \etaH\rangle_{\cE} -\langle\hatfY , \rD_{\xiH}\etaH\rangle_{\cE} - \langle\hatfY, \Gamma_{\cH}(\xiH, \etaH)\rangle_{\cE} =\\ \hatfYY(\xiH, \etaH) - \langle\hatfY, \Gamma_{\cH}(\xiH, \etaH)\rangle_{\cE}
\end{gathered}
$$ 
from compatibility with metric, idempotency of $\Pi_{\cH,\sfg}$, \cref{eq:useGammaH} and  \cref{eq:ehess}.\qed
\end{proof}
When the projection is given in terms of $\JJ$, and $\JJ$ is sufficiently smooth we have:
\begin{proposition}\label{prop:rhess}
If $\JJ$ as in \cref{prop:NJ} is of class $C^2$ then:
\begin{equation}
\label{eq:rhess02a}
\begin{gathered}
\Gamma_{\cH}(\xiH, \etaH) = \sfg^{-1}\JJ^{\ft}(\JJ \sfg^{-1}\JJ^{\ft})^{-1}(\rD_{\xiH}\JJ)\etaH+ \Pi_{\cH, \sfg}\sfg^{-1} \rK(\xiH, \etaH)
\end{gathered}
\end{equation}
for two $\cH$-valued tangent vectors $\xiH,\etaH$ at $Y\in \cM$. We can evaluate $\rhessHf^{11}\xiH$ by setting $\omega =\hatfY$ in the following formula, which is valid for all $\cE$-valued function $\omega$:
\begin{equation}
\label{eq:rhess11a}
\begin{gathered}
\nabla_{\xiH}\Pi_{\cH,\sfg}\sfg^{-1} \omega = \Pi_{\cH,\sfg} \sfg^{-1} \rD_{\xiH}\omega - \Pi_{\cH,\sfg}\sfg^{-1}(\rD_{\xiH} \sfg)\sfg^{-1}\omega -\\
\Pi_{\cH,\sfg} (\rD_{\xiH}(\sfg^{-1}\JJ^{\ft})) (\JJ\sfg^{-1}\JJ^{\ft})^{-1}\JJ\sfg^{-1}\omega+ \Pi_{\cH,\sfg}\sfg^{-1}\rK(\xiH, \Pi_{\cH,\sfg}(\sfg^{-1} \omega))
\end{gathered}
\end{equation}
\end{proposition}
\begin{proof}
The first expression follows by expanding $\rD_{\xiH}\Pi_{\cH,\sfg}$ in terms of $\JJ$, noting $\JJ\etaH = 0$. For the second, expand $\Pi_{\cH, \sfg}\hat{\nabla}_{\xiH}(\cH, \Pi_{\sfg}\sfg^{-1}\omega)=\Pi_{\cH, \sfg}\rD_{\xiH} (\Pi_{\cH,\sfg}\sfg^{-1}\omega)+ \Pi_{\cH,\sfg}\sfg^{-1}\rK(\xiH, \Pi_{\cH, \sfg}\sfg^{-1}\omega)$, then expand the first term and use $\Pi_{\cH, \sfg}\sfg^{-1}\JJ^{\ft} = 0$.\qed
\end{proof}
Recall (\cite{ONeil1983}, Definition 7.44) a Riemannian submersion   $\pi:\cM\to\cB$ between two manifolds $\cM$ and $\cB$ is a smooth, onto map, such that the differential $d\pi$ is onto at every point $Y\in\cM$, the fiber $\pi^{-1}(b), b\in \cB$ is a Riemannian submanifold of $\cM$, and $d\pi$ preserves scalar products of vectors normal to fibers. An important example is the quotient space by a free and proper action of a group of isometries. At each point $Y\in \cM$, the tangent space of $\pi^{-1}(\pi Y)$ is called the vertical space, and its orthogonal complement with respect to the Riemannian metric is called the horizontal space. The collection of horizontal spaces $\cH_Y$ ($Y\in\cM$) of a submersion is a subbundle $\cH$. The {\it horizontal lift}, identifying a tangent vector $\xi$ at $b=\pi(Y)\in \cB$ with a horizontal tangent vector $\xiH$ at $Y$ is a linear isometry between the tangent space $T_b\cB$ and $\cH_Y$, the horizontal space at $Y$. The following proposition allows us to apply the results so far in familiar contexts:
\begin{proposition}\label{prop:ambAppl}
  Let $(\cM, \sfg, \cE)$ be an embedded ambient structure.
  \begin{itemize}    
\item[1.] Fix an orthogonal basis $e_i$ of $\cE$, let $f$ be a function on $\cM$, which is a restriction of a function $\hatf$ on $\cE$, define $\hatfY$ to be the function from $\cM$ to $\cE$, having the $i$-th component the directional derivative $\rD_{e_i}\hatf$, then $\hatfY$ is an ambient gradient. If $\cM$ is defined by the equation $\bC(\bY) = 0$ ($Y\in \cM$) with a full rank Jacobian, then the nullspace of the Jacobian $\JJ_{\bC}(Y)$ is the tangent space of $\cM$ at $Y$, hence $\JJ_{\bC}(Y)$ could be used as the operator $\JJ(Y)$.
\item[2.] (Riemannian submersion) Let $(\cM, \sfg, \cE)$ be an embedded ambient structure.
  Let $\pi:\cM\to\cB$ be a Riemannian submersion, with $\cH$ the corresponding horizontal subbundle of $T\cM$. If $\xi, \eta$ are two vector fields on $\cB$ with $\xiH, \etaH$ their horizontal lifts, then the Levi-Civita connection $\nabla^{\cB}_{\xi}\eta$ on $\cB$ lifts to $\Pi_{\cH,\sfg}\nabla_{\xiH}\etaH$, hence \cref{eq:useGammaH} applies. Also, Riemannian gradients and Hessians on $\cB$ lift to $\cH$-Riemannian gradients and Hessians on $\cM$.
  \end{itemize}
\end{proposition}
\begin{proof}
  The construction of $\hatfY$ ensures $\langle\hatfY, e_i\rangle_{\cE} = \rD_{e_i}f$. The statement about the Jacobian is simply the implicit function theorem. Isometry of horizontal lift and \cite{ONeil1983}, Lemma 7.45, item 3, gives us Statement 2.\qed
\end{proof}

In practice, $\hatfY$ is computed by index-raising the directional derivative. For clarity, so far we use the subscript $\cH$ to indicate the relation to a subbundle $\cH$. For the rest of the paper, we will drop the subscripts $\cH$ on vector fields
(referring to $\xi$ instead of $\xi_{\cH}$) as it will be clear from the context if we discuss a vector field in $\cH$, or just a regular vector field.
\section{An example}
Let $\cM$ be a submanifold of $\cE$, defined by a  system of equations $\bC(\bx)=0$, where $\bC$ is a map from $\cE$ to $\R^k$ ($x\in\cM$). In this case, $\JJ_{\bC}=\bC_{\bx}$ is the Jacobian of $\bC$, assumed to be of full rank. The projection of $\omega\in\cE$ to the tangent space $T_x\cM$ given by
\begin{equation}
\Pi_{\sfg}\omega = \omega -\bC_{\bx}^{\mathsf{T}}(\bC_{\bx}\bC_{\bx}^{\mathsf{T}})^{-1}\bC_{\bx}\omega
\end{equation}
and the covariant derivative is given by $\nabla_{\xi}\eta= \rD_{\xi}\eta +
\bC_{\bx}^{\mathsf{T}}(\bC_{\bx}\bC_{\bx}^{\mathsf{T}})^{-1}(\rD_\xi\bC_{\bx})\eta$ for two vector fields $\xi, \eta$. With $\Gamma(\xi, \eta) = \bC_{\bx}^{\mathsf{T}}(\bC_{\bx}\bC_{\bx}^{\mathsf{T}})^{-1}(\rD_\xi\bC_{\bx})\eta$, the Riemannian Hessian bilinear form is computed from \cref{eq:rhess02}, and the Riemannian Hessian operator is:
$$\Pi_{\sfg}(\hatf_{xx}\xi-(\rD_\xi\bC_{\bx})^{\mathsf{T}}(\bC_{\bx}\bC_{\bx}^{\mathsf{T}})^{-1}\bC_{\bx}\hatf_x )$$
The expression $(\bC_{\bx}\bC_{\bx}^{\mathsf{T}})^{-1}\bC_{\bx}\hatf_x$ is often used as an estimate for the Lagrange multiplier, this result was discussed in section 4.9 of \cite{Edelman_1999}. When $C(x) = \bx^{\sfT}\bx - 1$ (the unit sphere) $\JJ_{C}\omega = \bx^{\sfT}\omega$, the Riemannian connection is thus $\nabla_{\xi}\eta= \rD_{\xi}\eta +\bx\xi^{\sfT}\eta$, a well-known result.

Our main interest is to study matrix manifolds. As seen, we need to compute $\rN^{\ft}$ or $\JJ^{\ft}$. We will review adjoint operators for basic matrix operations.
\section{Matrix manifolds: inner products and adjoint operators}
\subsection{Matrices and adjoints}\label{subsec:matrix}
We will use the trace (Frobenius) inner product on matrix vector spaces considered here.
Again, the base field $\KK$ is either $\R$ or $\C$. We use the letters $m, n, p$ to denote dimensions of vector spaces. We will prove results for both the real and complex cases together, as often there is a complex result using the Hermitian transpose corresponding to a real result using the real transpose. The reason is when $\C^{n\times m}$, as a real vector space, is equipped with the real inner product $\Real\Tr(ab^{\mathsf{H}})$ (for $a, b\in\C^{n\times m}$, $\mathsf{H}$ is the Hermitian transpose), then the adjoint of the scalar multiplication operator by a complex number $c$, is the multiplication by the conjugate $\bar{c}$.

To fix some notations, we use the symbol $\TrR$ to denote the real part of the trace, so for a matrix $a\in \KK^{n\times n}$, $\TrR a = \Tr a$ if $\KK=\R$ and $\TrR a  = \Real(\Tr a)$ if $\KK=\C$. The symbol $\ft$ will be used on either an operator, where it specifies the adjoint with respect to these inner products, or to a matrix, where it specifies the corresponding adjoint matrix. When $\KK=\R$, we take $\ft$ to be the real transpose, and when $\KK=\C$ we take $\ft$ to be the complex transpose. The inner product of two matrices $a, b$ is $\TrR(ab^{\ft})$. Recall that we denote by $\Herm{\ft}{\KK}{n}$ the space of all $\ft$-symmetric matrices ($A^{\ft} = A$), and $\AHerm{\ft}{\KK}{n}$ the space of all $\ft$-antisymmetric matrices ($A^{\ft} = -A$). We consider both $\Herm{\ft}{\KK}{n}$ and $\AHerm{\ft}{\KK}{n}$ inner product spaces under $\TrR$. We defined the symmetrizer $\sym{\ft}$ and antisymmetrizer $\asym{\ft}$ in \cref{sec:notation}, with the usual meaning.
\begin{proposition}\label{prop:raising} With the above notations, let $A_i, B_i, X$ be matrices such that the functional $L(X)= \sum_{i=1}^k \TrR(A_i X B_i) + \TrR(C_i X^{\ft} D_i)$ is well-formed. We have:\hfill\break
1. The matrix $\xtrace(L, X) = \sum_{i=1}^k A_i^{\ft}B_i^{\ft} + D_i C_i$ is the unique matrix $L_1$ such that $\TrR L_1 X^{\ft}=L(X)$ for all $X\in\KK^{n\times m}$ (this is the gradient of $L$).\hfill\break
2. The matrix  $\xtrace^{\mathrm{sym}}(L, X) = \sym{\ft}(\sum_{i=1}^k A_i^{\ft}B_i^{\ft} + D_i C_i)$ is the unique matrix $L_2\in\Herm{\ft}{\KK}{n}$ satisfying $\TrR(L_2X^{\ft}) = L(X)$ for all $X\in\Herm{\ft}{\KK}{n}$.\hfill\break
3. The matrix $\xtrace^{\mathrm{skew}}(L, X) = \asym{\ft}(\sum_{i=1}^k A_i^{\ft}B_i^{\ft} + D_i C_i)$ is the unique matrix $L_3\in\AHerm{\ft}{\KK}{n}$ satisfying $\TrR(L_3X^{\ft}) = L(X)$ for all $X\in\AHerm{\ft}{\KK}{n}$.
\end{proposition}
There is an abuse of notation as $\xtrace(L, X)$ is not a function of two variables, but $X$ should be considered a (symbolic) variable and $L$ is a function in $X$, however, this notation is convenient in symbolic implementation.
\begin{proof} Since $\TrR(A_i X B_i) = \TrR(B_i^{\ft}X^{\ft}A^{\ft}_i) = \TrR(A^{\ft}_iB_i^{\ft}X^{\ft})$ and $\TrR(C_i X^{\ft} D_i) = \TrR( D_iC_i X^{\ft})$, we have $\TrR(\xtrace(L)X^{\ft}) = L(X)$. Uniqueness follows from the fact that $\TrR$ is a non-degenerate bilinear form. The last two statements follow from $\TrR(\xtrace(L)X^{\ft})=\TrR(\xtrace(L)^{\ft})X$ if $X^{\ft} = X$ and $\TrR(\xtrace(L)^{\ft}X)=-\TrR(\xtrace(L)^{\ft})X$ if $X^{\ft} = -X$.
\end{proof}
\begin{remark}\label{rem:xtrace} The index-raising operation\slash gradient $\xtrace$ could be implemented as a symbolic operation on matrix trace expressions, as it involves only linear operations, matrix transpose, and multiplications. It could be used to compute an ambient gradient, for example. For another application, let $\cM$ be a manifold with ambient space $\KK^{n\times m}$, recall $\rhess_f^{02}(\xi, \eta) = \hatfYY(\xi, \eta) -\Gamma(\xi, \eta)$. Assume $\langle\Gamma(\xi, \eta),\hatfY\rangle_{\cE} = \sum_i\TrR(A_i \eta B_i) + \TrR(C_i \eta^{\ft} D_i)$ with $A_i, B_i, C_i, D_i$ are not dependent on $\eta$, and identify tangent vectors with their images in $\KK^{n\times m}$, we have:
  $$\rhess_f^{11}\xi = \Pi_{\sfg}\sfg^{-1}\xtrace(\rhess_f^{02}(\xi, \eta), \eta)$$
as the inner product of the right-hand side with $\eta$ is $\rhess_f^{02}(\xi, \eta)$, and the projection ensures it is in the tangent space. If the ambient space is identified with $\Herm{\ft}{\KK}{n}$, $\AHerm{\ft}{\KK}{n}$ or a direct sum of matrix spaces, we also have similar statements.
\end{remark}
\begin{proposition}\label{prop:adjoint} With the same notations as \cref{prop:raising}:\hfill\break
1. The adjoint of the left multiplication operator by a matrix $A\in \KK^{m\times n}$, sending $X\in \KK^{n\times p}$ to $AX\in\KK^{m\times p}$ is the left multiplication by $A^{\ft}$, sending  $Y\in \KK^{m\times p}$ to $A^{\ft}Y \in \KK^{m\times n}$.\hfill\break
2. The adjoint of the right multiplication operator by a matrix $A\in \KK^{m\times n}$ from $\KK^{p\times m}$ to $\KK^{p\times n}$ is the right multiplication by $A^{\ft}$.\hfill\break
3. The adjoint of the operator sending $X\mapsto X^{\ft}$ for $X\in\KK^{m\times m}$ is again the operator $Y\mapsto Y^{\ft}$ for $Y\in \KK^{m\times m}$. Adjoint is additive, and $(F\circ G)^{\ft} = G^{\ft}\circ F^{\ft}$ for two linear operators $F$ and $G$.\hfill\break
4. The adjoint of the left multiplication operator by $A\in \KK^{p\times n}$ sending $X\in \Herm{\ft}{\KK}{p}$ to $AX\in \KK^{p\times n}$ is the operator sending $Y\mapsto \frac{1}{2}(A^{\ft}Y+Y^{\ft}A)$ for $Y\in \KK^{p\times n}$. Conversely, the adjoint of the operator $Y\mapsto \frac{1}{2}(A^{\ft}Y+Y^{\ft}A)\in \Herm{\ft}{\KK}{p}$ is the operator $X\mapsto AX$.\hfill\break
5. The adjoint of the left multiplication operator by $A\in \KK^{p\times n}$ sending $X\in \AHerm{\ft}{\KK}{p}$ to $AX\in \KK^{p\times n}$ is the operator sending $Y\mapsto \frac{1}{2}(A^{\ft}Y-Y^{\ft}A)$ for $Y\in \KK^{p\times n}$. Conversely, the adjoint of the operator $Y\mapsto \frac{1}{2}(A^{\ft}Y-Y^{\ft}A)\in \AHerm{\ft}{\KK}{p}$ is the operator $X\mapsto AX$.
\hfill\break
6. Adjoint is linear on the space of operators. If $F_1$ and $F_2$ are two linear operators from a space $V$ to two spaces $W_1$ and $W_2$, then the adjoint of the direct sum operator (operator sending $X$ to $\begin{bmatrix}F_1 X & F_2 X\end{bmatrix}$) is the map sending $\begin{bmatrix}A\\ B\end{bmatrix}$ to $F_1^{\ft}A + F_2^{\ft}B$. Adjoint of the map sending $\begin{bmatrix}X_1 \\ X_2 \end{bmatrix}$ to $FX_1$ is the map  $Y\mapsto\begin{bmatrix}F^{\ft}Y \\ 0 \end{bmatrix}$, and more generally a map sending a row block $X_i$ of a matrix $X$ to $FX_i$ is the map sending $Y$ to a matrix where the $i$-th block is $F^{\ft}Y$, and zero outside of this block.
\end{proposition}
Most of the proof is just a simple application of trace calculus. For the first statement, the real case follows from $\Tr(Aab^{\mathsf{T}})=\Tr(a(A^{\mathsf{T}}b)^{\mathsf{T}})$, and $\TrR(Aab^{\mathsf{H}})=\TrR(a(A^{\mathsf{H}}b)^{\mathsf{H}})$ gives us the complex case. Statement 2. is proved similarly, statement 4 is standard. Statements 4. and 5. are checked by direct substitution, and 6. is just the operator version of the corresponding matrix statement, observing for example:
$$\TrR(F_1XA^{\ft} + F_2XB^{\ft}) = \TrR((F_1^{\ft}A + F_2^{\ft}B)X^{\ft})$$
\section{Application to Stiefel manifold}\label{sec:stiefel}
The Stiefel manifold $\St{\KK}{d}{n}$ on $\cE=\KK^{n\times d}$ is defined by the equation $\bY^{\ft}\bY = \bI_{d}$, where the tangent space at a point $\bY$ consists of matrices $\eta$ satisfying $\eta^{\ft}\bY + \bY\eta^{\ft} = 0$. We apply the results of \cref{sec:ambient} for the full tangent bundle $\cH=T\St{\KK}{d}{n}$. We can consider an ambient metric:
\begin{equation}
  \label{eq:stief_metric}
  \sfg(Y) \omega = \alpha_0\omega + (\alpha_1-\alpha_0)\bY\bY^{\ft}\omega = \alpha_0(I_n - \bY\bY^{\ft})\omega + \alpha_1 \bY\bY^{\ft}\omega
\end{equation}
for $\omega\in\cE=\KK^{n\times d}$. It is easy to see $\omega_0-\bY\bY^{\ft}\omega_0$ is an eigenvector of $\sfg(Y)$ with eigenvalue $\alpha_0$,  and $\bY\bY^{\ft}\omega_1$ is an eigenvector with eigenvalue $\alpha_1$, for any $\omega_0,\omega_1\in \cE$, and these are the only eigenvalues and vectors. Hence, $\sfg(Y)^{-1}\omega = 
\alpha_0^{-1}(I_n - YY^{\ft})\omega + \alpha^{-1}_1\bY\bY^{\ft}\omega$
and $\sfg$ is a Riemannian metric if $\alpha_0, \alpha_1$ are positive. We can describe the tangent space as a nullspace of $\JJ(Y)$ with $\JJ(Y)\omega = \omega^{\ft}\bY + \bY^{\ft}\omega\in \cE_{\JJ}:=\Herm{\ft}{\KK}{d}$. We will evaluate everything at $Y$, so we will write $\JJ$ and $\sfg$ instead of $\JJ(Y)$ and $\sfg(Y)$, etc. By \cref{prop:adjoint}, $\JJ^{\ft}\ba = (a\bY^{\ft})^{\ft} + \bY\ba=2\bY a$ for $\ba\in \cE_{\JJ}$. We have $\sfg^{-1}\JJ^{\ft}\ba=\alpha_0^{-1}2\bY\ba +(\alpha_1^{-1}-\alpha_0^{-1})2\bY\ba = 2\alpha_1^{-1}\bY\ba$. Thus $\JJ\sfg^{-1}\JJ^{\ft}\ba = \JJ(2\alpha_1^{-1}\bY\ba) =4\alpha_1^{-1}\ba$ and by \cref{prop:JN}:
\begin{equation}
\label{eq:st_proj_upper}
\Pi_{\sfg}\nu = \Pi_{\sfg}\nu = \nu -  2\alpha_1^{-1}\frac{\alpha_1}{4}(\bY\nu^{\ft}\bY + \bY\bY^{\ft}\nu)=\nu - \frac{1}{2}(\bY\nu^{\ft}\bY + \bY\bY^{\ft}\nu)\end{equation}
In this case, the ambient gradient $\hatfY$ is the matrix of partial derivatives of an extension of $f$ on the ambient space $\KK^{n\times d}$. More conveniently, using the eigenspaces of $\sfg$, $\Pi_{\sfg}\nu = (\bI_n-\bY\bY^{\ft})v + Y\asym{\ft}Y^{\ft}v$ and $\sfg^{-1}\hatfY = \alpha_0^{-1}(\bI_n-\bY\bY^{\ft})\hatfY +\alpha_1^{-1}\bY\bY^{\ft}\hatfY$
$$\begin{gathered}\Pi_{\sfg}\sfg^{-1} \hatfY =\alpha_0^{-1}(I_n - YY^{\ft})\hatfY +\alpha_1^{-1}Y\asym{\ft}(Y^{\ft}\hatfY)\\= \alpha_0^{-1}\hatfY+\frac{\alpha_1^{-1}-2\alpha_0^{-1}}{2}\bY\bY^{\ft}\hatfY -\frac{\alpha_1^{-1}}{2}\bY \hatfY^{\ft}\bY
\end{gathered}$$
If $\xi$ and $\eta$ are vector fields, $(\rD_{\xi}\sfg) \eta = (\alpha_1-\alpha_0)(\xi\bY^{\ft}+\bY\xi^{\ft})\eta$. Using \cref{prop:raising}, we can take the cross term $\cX(\xi, \eta) = (\alpha_1-\alpha_0)(\xi\eta^{\ft}+\eta\xi^{\ft})\bY$, thus:
$$\rK(\xi, \eta) = \frac{\alpha_1 - \alpha_0}{2}((\xi\bY^{\ft}\eta +\eta\bY^{\ft}\xi)+\bY(\xi^{\ft}\eta+\eta^{\ft}\xi) -(\xi\eta^{\ft}+\eta\xi^{\ft})\bY)$$
By the tangent condition, $(\xi\bY^{\ft}\eta +\eta\bY^{\ft}\xi)= -(\xi\eta^{\ft}+\eta\xi^{\ft})\bY$, hence $\rK(\xi, \eta) = \frac{\alpha_1-\alpha_0}{2}F$ with $F = Y(\xi^{\ft}\eta + \eta^{\ft}\xi) - 2(\xi^{\ft}\eta + \eta^{\ft}\xi)Y$, we see $Y^{\ft}F$ is symmetric so $\asym{\ft}Y^{\ft}F = 0$, therefore
$$\Pi_{\sfg}\sfg^{-1}F = \alpha_0^{-1}(I_n-YY^{\ft})F = - 2\alpha_0^{-1}(I_n-YY^{\ft})(\xi^{\ft}\eta + \eta^{\ft}\xi)Y$$
\begin{equation}\label{eq:pisfgK}\Pi_{\sfg}\sfg^{-1}\rK(\xi, \eta)=\frac{\alpha_0-\alpha_1}{\alpha_0}(\bI_n-\bY\bY^{\ft})(\xi\eta^{\ft}+\eta\xi^{\ft})\bY
\end{equation}
Using $\sfg^{-1}\JJ^{\ft}(\JJ\sfg^{-1}\JJ^{\ft})^{-1}(\rD_\xi\JJ)\eta = \frac{1}{2}\bY(\xi^{\ft}\eta+\eta^{\ft}\xi)$ to evaluate \cref{eq:rhess02a}, the connection for two vector fields $\xi, \eta$ is:
\begin{equation}
  \label{eq:stiefel_connect}
\nabla_{\xi}\eta= \rD_{\xi}\eta +\frac{1}{2}\bY(\xi^{\ft}\eta+\eta^{\ft}\xi) +\frac{\alpha_0-\alpha_1}{\alpha_0}(\bI_n-\bY\bY^{\ft})(\xi\eta^{\ft}+\eta\xi^{\ft})\bY
\end{equation}
With $\Pi_0 = (\bI_n-\bY\bY^{\ft})$ and let $\hatfYY$ be the ambient Hessian, from \cref{eq:rhess02}:
\begin{equation}
\rhess^{02}_f(\xi, \eta)=\hatfYY(\xi,\eta) -\TrR(\hatfY^{\ft}\{\frac{1}{2}\bY(\xi^{\ft}\eta+\eta^{\ft}\xi) +\frac{\alpha_0-\alpha_1}{\alpha_0}\Pi_0(\xi\eta^{\ft}+\eta\xi^{\ft})\bY\})
\end{equation}
$\rhess^{11}_f\xi$ is $\Pi_{\sfg}\sfg^{-1}\xtrace(\rhess^{02}_{\xi, \eta}, \eta)$ by \cref{rem:xtrace}:
\begin{equation}
\rhess^{11}_f\xi =\Pi_{\bY, \sfg}\sfg^{-1}(\hatfYY\xi -\frac{1}{2}\xi(\hatfY^{\ft}\bY +\bY^{\ft} \hatfY) -\frac{\alpha_0-\alpha_1}{\alpha_0}(\Pi_0 \hatfY\bY^{\ft}+\bY \hatfY^{\ft}\Pi_0)\xi)
\end{equation}
We note the term inside $\Pi_{\bY, \sfg}\sfg^{-1}$ is not unique as it can be modified by any expression sent to zero by $\Pi_{\bY, \sfg}\sfg^{-1}$. The case $\alpha_0 = 1, \alpha_1=\frac{1}{2}$ correspond to the canonical metric on a Stiefel manifold, where the connection is given by formula 2.49 of \cite{Edelman_1999}, in a slightly different form, but we could show they are the same by noting $\bY\bY^{\ft}(\xi\eta^{\ft}+\eta\xi^{\ft})\bY=\bY(\xi^{\ft}\bY\bY^{\ft}\eta + \eta^{\ft}\bY\bY^{\ft}\xi)$ using the tangent constraint. The case $\alpha_0 = \alpha_1 = 1$ corresponds to the constant trace metric where we do not need to compute $\rK$. This family of metrics has been studied in \cite{ExtCurveStiefel}, where a closed-form geodesic formula is provided. In \cite{NguyenGeodesic} we also provide computationally efficient closed-form geodesic formulas similar to those in \cite{Edelman_1999}.
\section{Quotients of a Stiefel manifold and flag manifolds}\label{sec:flag_manifold}
We will treat families of quotients of a Stiefel manifold, slightly more general than flag manifolds. Background materials for optimization on flag manifolds are in \cite{Nishi,YeLim}, but the review below should be sufficient to understand the setup and the results. We generalize the formula for $\rhess^{02}$ in \cite[Proposition 25]{YeLim} to the full family of metrics in \cite{ExtCurveStiefel} and provide a formula for $\rhess^{11}$.

Continuing with the setup in the previous section, consider a Stiefel manifold $\St{\KK}{d}{n}$ (we will assume $0 <d < n$). The metric induced by the operator $\sfg$ in \cref{eq:stief_metric}, $\alpha_0\TrR\omega_1^{\ft}\omega_2 + (\alpha_1-\alpha_0)\Tr\omega_1^{\ft}YY^{\ft}\omega_2$ with $Y\in\St{\KK}{d}{n}$, $\omega_1, \omega_2\in\cE$ is preserved if we replace $Y, \omega_1, \omega_2$ by $YU, \omega_1U, \omega_2U$, for $U^{\ft}U = I_d$, or if we define the $\ft$-orthogonal group by $\UU{\KK}{d} := \{U\in \KK^{d\times d}| U^{\ft}U = I_d\}$ then this is a group of isometries of $\sfg$. Therefore, any subgroup $G$ of $\UU{\KK}{d}$ acts on $\St{\KK}{d}{n}$ by right-multiplication also preserves the metric, and if $G$ is compact, we can consider the quotient manifold $\St{\KK}{d}{n}/G$, identifying $Y\in\St{\KK}{d}{n}$ with $YU$ for $U\in G$. The quotient manifold is useful for optimization problems with cost functions invariant under the action of the group $G$. In that case, instead of optimization over $\St{\KK}{d}{n}$, we can optimize over $\St{\KK}{d}{n}/G$. The case of the Rayleigh quotient cost function $\Tr Y^{\ft}A Y$ for $Y\in\St{\KK}{p}{n}$, $A$ is a positive-definite matrix is well-known, as the cost function is invariant under $G=\UU{\KK}{d}$, we can optimize over the Grassmann manifolds $\St{\KK}{d}{n}/G$. This section deals with the optimization of cost functions with a smaller group of symmetries $G$ described below.

We will consider the case where $G$ is a group of block-diagonal matrices. Assume there is a sequence of positive integers $\bdh = \{d_1,\cdots, d_q\}$, $d_i > 0$ for $1\leq i\leq q$, such that $\sum_{i=1}^q d_i \leq d$. Set $d_{q+1} = d - \sum_{i=1}^q d_i$, thus $d_{q+1}\geq 0$. This sequence allows a partition of a matrix $A\in\KK^{d\times d}$ to $(q+1)\times (q+1)$ blocks $A_{[ij]}\in\KK^{d_i\times d_j}$, $1\leq i, j\leq q+1$. The right-most or bottom blocks correspond to $i$ or $j$ equals to $q+1$ are empty when $d_{q+1} = 0$. Consider the subgroup $G=\UU{\KK}{\bdh}= \UU{\KK}{d_1}\times \UU{\KK}{d_2}\cdots\times \UU{\KK}{d_q}\times \{I_{d_{q+1}}\}$ of $\UU{\KK}{d}$ of block-diagonal matrices $U$, with the $i$-th diagonal block from the top $U_{[ii]}\in\UU{\KK}{d_i}$, $1\leq i\leq q$, and $U_{[q+1, q+1]} = I_{d_{q+1}}$. An element $U\in\UU{\KK}{\bdh}$ has the form
$$ U =\diag(U_{[11]}, \cdots, U_{[qq]}, I_{d_{q+1}})\text{ for } U_{[ii]}\in \UU{\KK}{d_i}, 1\leq i\leq q$$
When $q=0$, we define $\bdh=\emptyset$ and $\UU{\KK}{\bdh}= \{I_{d}\}$. We will consider the manifold $\St{\KK}{d}{n}/G$ for $G=\UU{\KK}{\bdh}$. Thus, when $\bdh = \emptyset$, this quotient is the Stiefel manifold $\St{\KK}{d}{n}$ itself, when $\bdh = \{d\}$, it is the Grassmann manifold. When $d_{q+1} = 0$ i.e. $\sum_{i=1}^q d_i = d$, the quotient is called a flag manifold, denoted by $\Flag(d_1,\cdots, d_q; n, \KK)$. Thus, these quotients could be considered as intermediate objects between a Stiefel and a Grassmann manifold, as we will soon see more clearly.

Define the operator $\fsym$ acting on $\KK^{d\times d}$, sending $A\in \KK^{d\times d}$ to $A_{\fsym}$ such that $(A_{\fsym})_{[ij]} = \frac{1}{2}(A_{[ij]} +A_{[ji]}^{\ft})$ if $1\leq i\neq j\leq q+1$ or $i=j=q+1$, and $(A_{\fsym})_{[ii]} = A_{[ii]}$ if $1 \leq i\leq q$. Thus, $\fsym$ preserves the diagonal blocks for $1\leq i \leq q$, but symmetrizes the off-diagonal blocks and the $q+1$-th diagonal block. The following illustrates the operation when $q=2$ for $A = A_{[ij]}\in \KK^{d\times d}$.
$$A_{\fsym} = \frac{1}{2}\begin{bmatrix}2A_{[11]} & A_{[12]} + A_{[21]}^{\ft} & A_{[13]} + A_{[31]}^{\ft} \\
  A_{[21]} + A_{[12]}^{\ft}  & 2A_{[22]} & A_{[23]} + A_{[32]}^{\ft}\\
  A_{[31]} + A_{[13]}^{\ft} & A_{[32]} + A_{[23]}^{\ft} & A_{[33]} + A_{[33]}^{\ft}
  \end{bmatrix}
$$
For the case $\bdh=\emptyset$ of the full Stiefel manifold, $\fsym$ is just $\sym{\ft}$ and for the case $\bdh=\{d\}$ of the Grassmann manifold, $\fsym$ is the identity map. We show these quotients, including Stiefel, Grassmann and flag manifolds, share similiar Riemannian optimization settings.
\begin{theorem}With the metric in \cref{eq:stief_metric}, the horizontal space $\cH_Y$ at $Y\in\St{\KK}{d}{n}$ of the quotient $\St{\KK}{d}{n}\to\St{\KK}{d}{n}/\UU{\KK}{\bdh}$ consists of matrices $\omega \in \cE:=\KK^{n\times d}$ such that
  \begin{equation}  (Y^{\ft}\omega)_{\fsym} = 0\end{equation}
    or equivalently, $Y^{\ft}\omega$ is $\ft$-antisymmetric with first $q$ diagonal blocks $((Y^{\ft}\omega)_{\fsym})_{[ii]}$ vanish for $1\leq i\leq q$. For $\omega\in\cE=\KK^{n\times d}$, the projection $\Pi_{\cH}$ from $\cE$ to $\cH_Y$ and the Riemannian gradient are given by
\begin{gather}  \Pi_{\cH}\omega = \omega - Y(Y^{\ft}\omega)_{\fsym}\label{eq:flag_proj}\\
  \Pi_{\cH}\sfg^{-1}\hatfY = \alpha_0^{-1}\hatfY + (\alpha_1^{-1} - \alpha_0^{-1})YY^{\ft}\hatfY - \alpha_1^{-1}Y(Y^{\ft}\hatfY)_{\fsym}\label{eq:flag_rgrad}
\end{gather}
Let $\Pi_0 = I_n - YY^{\ft}$. For two vector fields $\xi, \eta$, the horizontal lift of the Levi-Civita connection and Riemannian Hessians are given by
\begin{gather}
  \label{eq:flag_connect}
\Pi_{\cH}\nabla_{\xi}\eta= \rD_{\xi}\eta +\bY(\xi^{\ft}\eta)_{\fsym} +\frac{\alpha_0-\alpha_1}{\alpha_0}\Pi_0(\xi\eta^{\ft}+\eta\xi^{\ft})\bY\\
\rhess^{02}_f(\xi, \eta)=\hatfYY(\xi,\eta) -\TrR(\hatfY^{\ft}\{\bY(\xi^{\ft}\eta)_{\fsym} +\frac{\alpha_0-\alpha_1}{\alpha_0}\Pi_0(\xi\eta^{\ft}+\eta\xi^{\ft})\bY\})\label{eq:flag_rhess02}\\
\rhess^{11}_f\xi =\Pi_{\cH}\sfg^{-1}(\hatfYY\xi -\xi(\bY^{\ft}\hatfY)_{\fsym} -\frac{\alpha_0-\alpha_1}{\alpha_0}(\Pi_0 \hatfY\bY^{\ft}+\bY \hatfY^{\ft}\Pi_0)\xi)\label{eq:flag_rhess11}
\end{gather}
\end{theorem}
\begin{proof}First we note that $\fsym$ is a self-adjoint operator, as both the identity
 operator on the first $q$ diagonal blocks and symmetrize operator on the remaining blocks are self-adjoint. The orbit of $Y\in\St{\KK}{d}{n}$ under the action of $\UU{\KK}{\bdh}$ is $Y\UU{\KK}{\bdh}$, thus the vertical space consists of matrices of the form $YD$ with $D$ is block-diagonal, $\ft$-skewsymmetric and $D_{[(q+1),(q+1)]} = 0$. Since $\sfg YD = \alpha_1 YD$, a horizontal vector $\omega$ satisfies $\asym{\ft}(Y^{\ft}\omega) = 0$ and $\TrR(\omega^{\ft}YD) = 0$. This shows $Y^{\ft}\omega$ has zero first $q$ diagonal blocks, hence $(Y^{\ft}\omega)_{\fsym} = 0$.

For the projection, we proceed like the Stiefel case, with the map $\JJ\omega = (Y^{\ft}\omega)_{\fsym}$, mapping $\cE$ to $\cE_{\JJ}= \{ A\in \KK^{d\times d}| A_{[ij]} =A_{[ji]}^{\ft}, 1\leq i \neq j\leq q+1 \text{ or } i=j=q+1\}$. Since $\fsym$ is self-adjoint, $\JJ^{\ft} A = YA_{\fsym} = YA$ for $A\in \cE_{\JJ}$. From here we get $(\JJ\sfg^{-1}\JJ^{\ft})A = \alpha_1^{-1}A$ and \cref{eq:flag_proj} follows. Equation \ref{eq:flag_rgrad} is a substitution of $\sfg^{-1}\hatfY$ to \cref{eq:flag_proj}, noting $(Y^{\ft}(\sfg^{-1}\hatfY))_{\fsym} = \alpha_1^{-1}(Y^{\ft}\hatfY)_{\fsym}$, using the eigen decomposition of $\sfg$.

For the Levi-Civita connection, we use \cref{eq:useGammaH}. For two horizontal vector fields $\xi, \eta$, $(\rD_{\xi}\Pi_{\cH})\eta = -\xi(Y^{\ft}\eta)_{\fsym} - Y(\xi^{\ft}\eta)_{\fsym}=-Y(\xi^{\ft}\eta)_{\fsym}$ and $\Pi_{\cH}=\Pi_{\cH}\Pi$, where $\Pi$ is the Stiefel projection \cref{eq:st_proj_upper}, hence $\Pi_{\cH}\sfg^{-1}\rK(\xi, \eta) = \Pi\sfg^{-1}\rK(\xi, \eta) - Y(Y^{\ft}\Pi\sfg^{-1}\rK(\xi, \eta))_{\fsym}$. The last term vanishes from \cref{eq:pisfgK}, and we have \cref{eq:flag_connect}.

Equation (\ref{eq:flag_rhess02}) follows from \cref{eq:rhess02}. We derive \cref{eq:flag_rhess11} from \cref{rem:xtrace} and self-adjointness of $\fsym$, expanding $\TrR\hatfY Y(\xi^{\ft}\eta)_{\fsym}$ to
$$\TrR(Y^{\ft}\hatfY)^{\ft}(\xi^{\ft}\eta)_{\fsym} = \TrR(Y^{\ft}\hatfY)_{\fsym}(\xi^{\ft}\eta)^{\ft}=\TrR\xi(Y^{\ft}\hatfY)_{\fsym}\eta^{\ft}$$
\end{proof}
\section{Positive-definite matrices}
Consider the manifold $\Sd{\KK}{n}$ of ${\ft}$-symmetric\break positive-definite matrices in $\KK^{n\times n}$. In our approach, we take $\cE=\KK^{n\times n}$ with its Frobenius inner product. The metric is $\langle \xi, \sfg\eta\rangle_{\cE} = \Tr(\xi\bY^{-1}\eta\bY^{-1})$, with the metric operator $\sfg:\eta\mapsto \bY^{-1}\eta\bY^{-1}$ for two vector fields $\xi, \eta$. The full tangent bundle $\cH=T\Sd{\KK}{n}$ is identified fiber-wise with the nullspace of the operator $\JJ:\eta\mapsto \JJ\eta = \eta - \eta^{\ft}$, with $\cE_{\JJ} = \AHerm{\ft}{\KK}{n}$. By item 5 in \cref{prop:adjoint}, we have $\JJ^{\ft} \ba = 2\ba$ where $a$ is a $\ft$-antisymmetric matrix. From here $\JJ\sfg^{-1}\JJ^{\ft}\ba = 4\bY\ba\bY$ and (write $\Pi$ for $\Pi_{\sfg}$):
$$\Pi\eta = \eta - \sfg^{-1}\JJ^{\ft}(\JJ\sfg^{-1}\JJ^{\ft})^{-1}\JJ\eta =\eta- 2\bY(\frac{1}{4}\bY^{-1}(\eta - \eta^{\ft})\bY^{-1})\bY = \frac{1}{2}(\eta+\eta^{\ft})=\sym{\ft}\eta$$
Thus, the Riemannian gradient is $\Pi\sfg^{-1}\hatfY=\frac{1}{2}\bY(\hatfY+\hatfY^{\ft})\bY$. Next, we compute
$(\rD_{\xi}\sfg)\eta = \rD_{\xi|\eta \text{ constant}}(\bY^{-1}\eta\bY^{-1}) = -\bY^{-1}\xi\bY^{-1}\eta\bY^{-1} - \bY^{-1}\eta\bY^{-1}\xi\bY^{-1}$, where we keep $\eta$ constant in the derivative, as we evaluate $\rD_{\xi}\sfg$ as an operator-valued function. From here, $(\rD_{\eta}\sfg)\xi =(\rD_{\xi}\sfg)\eta$. We note for three vector fields $\xi, \eta, \xi_0$
$$\Tr((\bY^{-1}\xi_0\bY^{-1}\eta\bY^{-1} + \bY^{-1}\eta\bY^{-1}\xi_0\bY^{-1})\xi)=\Tr(\xi_0(\bY^{-1}\eta\bY^{-1}\xi\bY^{-1} + \bY^{-1}\xi\bY^{-1}\eta\bY^{-1}))$$
Thus, we can take $\cX(\xi, \eta)=-\bY^{-1}\eta\bY^{-1}\xi\bY^{-1} - \bY^{-1}\xi\bY^{-1}\eta\bY^{-1}$ and
$$\Gamma(\xi, \eta) = -(\rD_{\xi}\Pi)\eta - \frac{1}{2}Y(\bY^{-1}\eta\bY^{-1}\xi\bY^{-1} + \bY^{-1}\xi\bY^{-1}\eta\bY^{-1})Y$$
\begin{equation}
\nabla_{\xi}\eta = \rD_{\xi}\eta -\frac{1}{2}(\xi\bY^{-1}\eta + \eta\bY^{-1}\xi)
\end{equation}
Thus the Riemannian Hessian bilinear form $\rhess^{02}(\xi, \eta)$ is 
\begin{equation}
\hatfYY(\xi, \eta) + \frac{1}{2}\Tr((\xi\bY^{-1}\eta + \eta\bY^{-1}\xi) \hatfY) = \hatfYY(\xi, \eta) + \frac{1}{2}\Tr((\hatfY \xi\bY^{-1} + \bY^{-1}\xi\hatfY)\eta)
\end{equation}
Using a symmetric version of \cref{rem:xtrace}, $\rhess^{11}_f\xi = \Pi_{\sfg}\sfg^{-1}\sym{\ft}(\hatfYY\xi + \frac{1}{2}(\hatfY \xi\bY^{-1} + \bY^{-1}\xi\hatfY))$. We get the following formula, as in \cite{manopt}:
\begin{equation}
\rhess_f^{11}\xi = \bY\sym{\ft}(\hatfYY \xi)\bY + \sym{\ft}(\xi\sym{\ft}(\hatfY) \bY)
\end{equation}
\section{A family of metrics for the manifold of positive-semidefinite matrices of fixed rank} \label{sec:psd}
In \cite{BonSel}, the authors defined a family of metrics on the manifold $\Sp{\KK}{p}{n}$ of positive-semidefinite matrices of size $n$ and rank $p$ for the case $\KK = \R$. Each such matrix will have the form $YPY^{\ft}$ with $Y\in \St{\KK}{p}{n}$ ($Y^{\ft}Y=I_p$) and $P$ is positive-definite of size $p\times p$, up to the equivalent relation $(Y, P)\sim (YU, U^{\ft}PU)$ for a matrix $U\in\UU{\KK}{p}$, (that means $U\in\KK^{p\times p}$ and $UU^{\ft} = I_p$). So the manifold $\Sp{\KK}{p}{n}$ could be identified with the quotient space $(\St{\KK}{p}{n}\times\Sd{\KK}{p})/\UU{\KK}{p}$ of the product of the Stiefel manifold $\St{\KK}{p}{n}$ and the manifold of positive-definite matrices $\Sd{\KK}{p}$ over the $\ft$-orthogonal group $\UU{\KK}{p}$. (The paper actually uses $R= P^{\frac{1}{2}}$ to parametrize the space.) From our point of view, the ambient space is $\cE = \KK^{n\times p}\times \KK^{p\times p}$, and the tangent space is identified with the image of the operator $\rN_1$ from $\KK^{(n-p)\times p}\times \Herm{\ft}{\KK}{p}$ to $\cE$, $\rN_1: (B, D) \mapsto (Y_{\perp} B, D)$, where the matrix $(Y| Y_{\perp})$ is $\ft$-orthogonal. On the tangent space, \cite{BonSel} uses the metric $\Tr(BB^{\ft})+k\Tr(DP^{-1}DP^{-1})$ for a positive number $k$. The action of the group gives us the vertical vectors $(Y\rmq, P\rmq - \rmq P)$ for a $\ft$-antisymmetric matrix $\rmq$ such that $\rmq+\rmq^{\ft} = 0$. In the paper, the image of $\rN_1$ transverses but not orthogonal to the vertical vectors and no second-order method is provided. We modify this approach, using a Riemannian quotient metric to provide a second-order method. In the following, the projection to the horizontal space is denoted by $\Pi_{\cH}$. The horizontal lift of the Levi-Civita connection is $\Pi_{\cH}\nabla$, which we will denote by $\nabla^{\cH}$.
\begin{theorem}\label{prop:psd_conn} Let $\cM = \St{\KK}{p}{n}\times\Sd{\KK}{p}\subset \cE:= \KK^{n\times p}\times \KK^{p\times p}$ be an embedded ambient space of $\cM$. Identifying the manifold $\Sp{\KK}{p}{n}$ of positive-semidefinite matrices with $\cB = (\St{\KK}{p}{n}\times\Sd{\KK}{p})/\UU{\KK}{p}$, where the pair $S =(Y, P)$ represents the matrix $YPY^{\ft}$ with $Y\in\St{\KK}{p}{n}, P\in \Sd{\KK}{p}$, and the action of $U\in\UU{\KK}{p}$ sends $(Y, P)$ to $(YU, U^{\ft}PU)$. The self-adjoint metric operator
  \begin{equation}\label{eq:psd_metric}\sfg(S)(\omega_Y, \omega_P)=\sfg(\omega_Y, \omega_P)=(\alpha_0\omega_Y + (\alpha_1 - \alpha_0)YY^{\ft}\omega_P, \beta P^{-1}\omega_P P^{-1})\end{equation}
  for $\omega = (\omega_Y, \omega_p)\in \cE= \KK^{n\times p}\times \KK^{p\times p}$ defines the inner product $\TrR(\alpha_0\omega_Y^{\ft}\omega_Y + (\alpha_1 - \alpha_0)\omega_Y^{\ft}YY^{\ft}\omega_Y +\beta \omega_PP^{-1}\omega_P P^{-1})$ on $\cE$, which induces a metric on $\St{\KK}{p}{n}\times\Sd{\KK}{p}$, invariant under the action of $\UU{\KK}{p}$ and induces a quotient metric on $\Sp{\KK}{p}{n}$. Its tangent bundle  $T\Sp{\KK}{p}{n}$ lifts to the subbundle $\cH\subset T(\St{\KK}{p}{n}\times\Sd{\KK}{p})$ horizontal to the group action, where a vector $\eta=(\eta_Y, \eta_P)\in \cE$ is a horizontal tangent vector at $S = (Y, P)$ if and only if it satisfies:
  \begin{equation}
    \label{eq:sp_horizontal} \alpha_1Y^{\ft}\eta_Y +\beta \eta_PP^{-1} - \beta P^{-1}\eta_P=0
    \end{equation}
  $\cH_S = \cH_{(Y, P)}$ could be identified as the range of the one-to-one operator $\rN(S)$ from $\cE_{\rN} = \KK^{(n-p)\times p}\times \Herm{\ft}{\KK}{p}$ to $\cE$, mapping $(B, D)\in \cE_{\rN}$ to:
  \begin{equation}
\rN(S)(B, D) = (\rN_Y(B, D), \rN_P(B, D)) :=  (\beta Y(P^{-1}D - D P^{-1}) + Y_{\perp}B, \alpha_1 D)
\end{equation}
where $Y_{\perp}$ is orthogonal complement matrix to $Y$, $(Y| Y_{\perp})\in \UU{\KK}{n}$. The projection of $\omega = (\omega_Y, \omega_P)\in \cE=\KK^{n\times p}\times \KK^{p\times p}$ to the horizontal space $\cH_{(Y, P)}$  is given by
\begin{equation}
  \label{eq:psd_proj}
  \begin{gathered}
    \Pi_{\cH}(S)(\omega_Y, \omega_P) = (\beta Y(P^{-1}\cD - \cD P^{-1}) +\omega_Y-YY^{\ft}\omega_Y,\alpha_1\cD)\\
    \text{with }\cD = \cD(P) \omega := L(P)^{-1}\sym{\ft}(\omega_P + Y^{\ft}\omega_YP - PY^{\ft}\omega_Y)\\
    \text{where }L(P)X := (\alpha_1 - 2\beta) X + \beta (P X P^{-1} + P^{-1} X P)
    \end{gathered}
\end{equation}
The operator $L(P)$ could be inverted by \cref{prop:ex_lyapunov}. The Riemannian Hessian could be evaluated by \cref{eq:rhess11}, and the lift of the Levi-Civita connection is given by
\begin{equation}\label{eq:nabla_psd}
  \begin{gathered}
 \nabla^{\cH}_{\xi}\eta := \Pi_{\cH}\nabla_{\xi}\eta =\rD_{\xi}\eta -(\rD_{\xi}\Pi_{\cH})\eta + \Pi_{\cH}\sfg^{-1}\rK(\xi, \eta)\\
  \end{gathered}
\end{equation}
for horizontal lifts $\xi=(\xi_Y, \xi_P), \eta = (\eta_Y, \eta_P)$ of tangent vectors of $\Sp{\KK}{p}{n}$ with
\begin{equation}
  \label{eq:KSpd}
\begin{gathered}
\rK(\xi, \eta)_Y = \frac{\alpha_1 - \alpha_0}{2}(Y (\eta_Y^{\ft} \xi_Y+\xi_Y^{\ft}\eta_Y) -2(\eta_Y\xi_Y^{\ft}+\xi_Y\eta_Y^{\ft})Y)\\
\rK(\xi, \eta)_P = - \frac{\beta}{2}(P^{-1} \eta_{P} P^{-1} \xi_{P} P^{-1} +  P^{-1} \xi_{P} P^{-1} \eta_{P} P^{-1})
\end{gathered}
\end{equation}
Set $\rcD = (\rD_{\xi}\cD)\omega = L(P)^{-1}\{ \sym{\ft}( \xi_Y^{\ft}\omega_YP - P\xi_Y^{\ft}\omega_Y +Y^{\ft}\omega_Y\xi_P-\xi_PY^{\ft}\omega_Y)\\
- \beta(\xi_P \cD P^{-1} + P^{-1}\cD \xi_P - P \cD P^{-1}\xi_PP^{-1} - P^{-1}\xi_PP^{-1}\cD P)\}$, we have
\begin{equation}\label{eq:DH_psd}
  \begin{gathered}
(\rD_{\xi}\Pi_{\cH})\omega = (((\rD_{\xi}\Pi_{\cH})\omega)_Y,\alpha_1\rcD)\\
((\rD_{\xi}\Pi_{\cH})\omega)_Y = \beta \xi_Y(P^{-1}\cD - \cD P^{-1}) + \beta Y(P^{-1}\rcD - \rcD P^{-1} +\\
    \cD P^{-1}\xi_P P^{-1}-P^{-1}\xi_P P^{-1}\cD) -(\xi_Y Y^{\ft}+ Y\xi_Y^{\ft})\omega_Y
\end{gathered}    
\end{equation}  
\end{theorem}
Thus, the Hessian could be evaluated at $O(np^2)$-complexity, by operator composition using \cref{eq:rhess11}. Note that $(\rD_{\xi}\Pi_{\cH})\eta$ could be further simplified if $\eta=(\eta_Y, \eta_P)$ is a horizontal tangent vector, as $\cD = \frac{1}{\alpha_1}\eta_P$ in that case.
\begin{proof} We have $\TrR(\alpha_1\eta_Y^{\ft}Y\rmq+ \beta(\eta_P^{\ft}P^{-1}(P\rmq - \rmq P)P^{-1})) = 0$ for a tangent vector $(\eta_Y, \eta_P)$ and a $\ft$-antisymmetric matrix $\rmq$ from the horizontal condition. Using \cref{prop:raising} this means $\asym{\ft}(\alpha_1\eta_Y^{\ft}Y +\beta P^{-1}\eta^{\ft}_P - \beta\eta_P^{\ft} P^{-1}) = 0$. Using the fact that $\eta_P^{\ft}=\eta_P$ and $\eta_Y^{\ft}Y$ is $\ft$-antisymmetric, we have \cref{eq:sp_horizontal}.
  
  It is clear that $\rN(B, D)$ satisfies this equation and is one-to-one: if $\rN(B, D) = 0$ then immediately $D=0$, and $B=0$ since $Y_{\perp}^{\ft}Y_{\perp}=I$. It is onto the tangent space by a dimension count. The adjoint $\rN^{\ft} =(\rN^{\ft}_B,  \rN^{\ft}_D)$ has two components corresponding to the $B$ and $D$ factors. By \cref{prop:adjoint} we have
  $$\begin{gathered} \rN^{\ft}(\omega_Y, \omega_P)_B = Y_{\perp}^{\ft} \omega_{Y}\\
 \rN^{\ft}(\omega_Y, \omega_P)_D = \sym{\ft}(\alpha_{1}\omega_{P}  + \beta P^{-1} Y^{\ft} \omega_{Y} - \beta Y^{\ft} \omega_{Y} P^{-1})\\
  \rN^{\ft}\sfg(\omega_Y, \omega_P)_B = \alpha_{0} Y_{\perp}^{\ft} \omega_{Y}\\
\rN^{\ft}\sfg(\omega_Y, \omega_P)_D= \alpha_{1} \beta\sym{\ft}(P^{-1} \omega_{P} P^{-1} +  P^{-1} Y^{\ft}\omega_{Y}  -  Y^{\ft} \omega_{Y} P^{-1})\\  
(\rN^{\ft}\sfg\rN(B, D))_B = \alpha_{0} B\\
  \end{gathered}$$
  $$\begin{gathered}
    (\rN^{\ft}\sfg\rN(B, D))_D = \alpha_{1} \beta\sym{\ft}(P^{-1}(\alpha_1D) P^{-1} +  P^{-1} Y^{\ft} (\beta Y(P^{-1}D - D P^{-1}) + Y_{\perp}B)  -\\  Y^{\ft} (\beta Y(P^{-1}D - D P^{-1}) + Y_{\perp}B) P^{-1})\\
    = \alpha_{1}\beta \sym{\ft}(\alpha_1P^{-1}D P^{-1}+  \beta P^{-2}D - \beta P^{-1}D P^{-1}  -  \beta P^{-1}D  P^{-1} +\beta D P^{-2})\\
= \alpha_1\beta P^{-1}((\alpha_1 -2\beta) D + \beta P D P^{-1}  +  \beta P^{-1} D P)P^{-1}\\
= \alpha_1\beta P^{-1}L(P)DP^{-1}
  \end{gathered}$$
  Hence $(\rN^{\ft}\sfg\rN)^{-1}(B, D) = (\alpha_0^{-1}B, (\alpha_1\beta)^{-1}L(P)^{-1}(PDP))$ and $ (\rN^{\ft}\sfg\rN)^{-1}\rN^{\ft}\sfg(\omega_Y, \omega_P) = (Y_{\perp}^{\ft}\omega_Y, \cD)$. The projection formula \cref{eq:psd_proj} is just $\rN(\rN^{\ft}\sfg\rN)^{-1}\rN^{\ft}\sfg$. The formulas for $\rK$ follow from the corresponding Stiefel and positive-definite manifold formulas. For $\rD_{\xi} \Pi_{\cH}$, we take directional derivative of \cref{eq:psd_proj} using standard matrix calculus, the only difficulty is $\rcD = (\rD_{\xi}\cD)\omega$. We evaluate it by evaluating $\rD_{\xi}(L(P)\cD)\omega$ as
$$\begin{gathered}
  (\alpha_1 - 2\beta) (\rcD + \beta (P \rcD  P^{-1} + P^{-1} \rcD P) +
  \beta(\xi_P\cD P^{-1} + P^{-1}\cD \xi_P) -\\ \beta(P \cD P^{-1}\xi_PP^{-1} + P^{-1}\xi_PP^{-1}\cD P)
 = \sym{\ft}( \xi_Y^{\ft}\omega_YP - P\xi_Y^{\ft}\omega_Y +Y^{\ft}\omega_Y\xi_P-\xi_PY^{\ft}\omega_Y)
  \end{gathered}$$
by differentiating the equation for $\cD$. From here, we get the equation for $\rcD$.\qed
\end{proof}
To solve for $\cD$ and $\rcD$, we need an extension of the symmetric Lyapunov equation:
\begin{proposition}\label{prop:ex_lyapunov}
Let $P\in \Herm{\ft}{\KK}{p}$ be an $\ft$-symmetric matrix having the eigenvalue decomposition $P = U\Lambda U^{\ft}$ with eigenvalues $(\Lambda_i)_{i=1}^p$ and $UU^{\ft}=I$. Let the set of coefficients $(c_{st})_{-k\leq s, t \leq k}$ be such that
$M_{ij} := \sum_{s=-k, t=-k}^{s=k, t=k}c_{st}\Lambda^s_i \Lambda^t_j$ is not zero for all pairs ($\Lambda_i, \Lambda_j)$ of eigenvalues of $P$, then the equation
\begin{equation}
  \label{eq:eq_lyapunov}
  \sum_{s=-k, t=-k}^{s=k, t=k} c_{st}P^{s} X P^{t} = B
\end{equation}
has the following unique solution that could be computed at $O(p^3)$ complexity:
\begin{equation}
  \label{eq:ex_lyapunov}
  X = U\{(U^{\ft}BU) / M\} U^{\ft}
\end{equation}  
with $M=(M_{ij})_{i=1,j=1}^{i=p, j=p}$ and $/$ denotes the by-entry division. In particular, for a positive-definite matrix $P$ and  positive scalars $(\alpha$, $\beta)$, the equation $(\alpha - 2\beta)X + \beta(P^{-1}XP + PXP^{-1})= B$ has a unique solution with $M_{ij} = \alpha +\beta(\Lambda_i^{-1}\Lambda_j + \Lambda_i\Lambda_j^{-1}-2)$.
\end{proposition}
\begin{proof} We follow the idea of \cite{BartelsStewart,BhatiaRos} but use the eigenvalue decomposition in place of the Schur decomposition. Substitute $P = U\Lambda U^{\ft}$ to \cref{eq:eq_lyapunov} and multiply $U^{\ft}$ and $U$ on the left-hand and right-hand sides of that equation we get $\sum_{s=-k, t=-k}^{s=k, t=k} c_{st}\Lambda^{s}U^{\ft} X U \Lambda^{t} = U^{\ft}B U$, which is equivalent to $(U^{\ft} X U)_{ij}M_{ij} = (U^{\ft}B U)_{ij}$ or $U^{\ft} X U = (U^{\ft}B U)/M$, and we have \cref{eq:ex_lyapunov}. If $\alpha$ and $\beta$ are positive then
  $\alpha_1 + \beta(\Lambda_i^{-1}\Lambda_j + \Lambda_i\Lambda_j^{-1}-2) >0$ by the $AGM$-inequality. The eigenvalue decomposition has $O(p^3)$ cost.\qed
\end{proof}
Horizontal lifts of geodesics on $\Sp{\KK}{p}{n}$ are geodesics on $\St{\KK}{p}{n}\times\Sd{\KK}{p}$. It is clear $\Sp{\KK}{p}{n}$ is a complete Riemannian manifold under this quotient metric, as both factors above are. If $(\eta_Y, \eta_P)$ is a horizontal tangent vector at $S=(Y, P)$, a horizontal geodesic $\gamma$ with $\gamma(0) = S, \dot{\gamma}(0) = (\eta_Y, \eta_P)$ is of the form $(Y(t), P^{1/2}\exp(t P^{-1/2}\eta_P P^{-1/2})P^{1/2})$, with $Y(t)$ is the geodesic of the metric in \cref{sec:stiefel}, described in \cite{NguyenGeodesic}.

\section{Implementation and numerical experiments} We developed a {\it Python} package based on {\it Pymanopt and Manopt} \cite{JMLR:v17:16-177,manopt} implementing the manifolds with metrics considered in this paper, for both the real and complex cases in the package \cite{Nguyen2020riemann}. The basic optimization algorithms are from \cite{JMLR:v17:16-177,manopt}, and the methods described in this paper provide the translation from ambient gradient and Hessian to the Riemannian ones. Besides the numerical implementation, we also include notebooks showing symbolic calculus results for each manifold, and scripts with numerical tests for the manifolds, including geodesics in most cases. We implemented the main formulas in \cref{prop:covar} and \cref{prop:rhess} as template methods in the base class {\it NullRangeManifold}. To use the method described here for a new manifold, it is convenient to {\it derive} from {\it NullRangeManifold}. Then, the user needs to implement the constraint operator $\JJ$, its transpose and derivative, the metric operator $\sfg$, its inverse, derivative, and the third Christoffel term $\cX$, as well as a method to solve $\JJ\sfg^{-1}\JJ^{\ft}$ (defaulted to use a conjugate-gradient solver otherwise). A retraction is also required. The template automatically provides the projection, Riemannian gradient, and Hessian derived in this paper.

We also implement real and complex Stiefel manifolds and positive-semidefinite manifolds with metric parameters. For each manifold, we provide a manifold class to support optimization problems based on the manifold. Using the method in this paper but with details appearing elsewhere, we also derive and implement optimization framework for fixed-rank matrix manifolds. For flag manifolds, we initially constructed a large family of metrics (including the metrics in \cref{sec:flag_manifold}) in \cite{Nguyen2020riemann}. We also implemented the metrics in \cref{sec:flag_manifold} separately in \cite{Nguyen2021riemannFlag}, which we use below. For testing, we numerically verify the projection satisfying the nullspace condition. We also tested metric compatibility and torsion-freeness of the covariance derivative and confirmed the relation between bilinear and operator Hessians. As the manifolds considered here are constructed from Stiefel or positive-definite matrix manifolds, both have effective retractions, we use these retractions in our implementation. As we would like to focus on methodology in this paper, we will not discuss formal numerical experiments. However, we have tested each manifold with a quadratic cost problem including matrices with one size of 1000 dimensions with a Trust-Region solver, which handles them comfortably.

For flag manifolds, we optimize the function $f(Y) = \TrR((Y\Lambda Y^{\ft}A)^2)$ over matrices $Y\in\St{\R}{d}{n}$. Here, $A$ is a positive-definite matrix, $d<n$ are two positive integers, $\bdh=(d_1,\cdots, d_q)$ is a partition of $d$ with $\sum_{i=1}^q d_i = d$, $\Lambda = \diag(\lambda_1 I_{d_1},\cdots, \lambda_q I_{d_q})$ for positive numbers $\lambda_1\cdots \lambda_q$. This cost function is invariant under the action of of $\UU{\R}{\bdh}$, thus could be considered as a function on the flag manifold $\St{\R}{d}{n}/\UU{\R}{\bdh}$. The Euclidean gradient is $4(AY\Lambda)Y^{\ft}(AY\Lambda)$, and the Euclidean Hessian is computed by routine matrix calculus. For testing, we consider $d=60, \bdh=(30, 20, 10), n=1000$ and use a trust-region solver. In this case, there is no noticeable variation in time when different values of $\alpha$ are used, typically a few seconds on a free colab engine in the notebook \href{https://github.com/dnguyend/SimpleFlag/blob/main/colab/SimpleFlag.ipynb}{colab/SimpleFlag.ipynb} in \cite{Nguyen2021riemannFlag}. Convergence is achieved after typically $16$ trust-region iterations, with small alpha requiring more outer iterations. The convergence is superlinear as seen in \cref{fig:FlagOpt}.
\begin{figure}
  \centering
\includegraphics[scale=0.4]{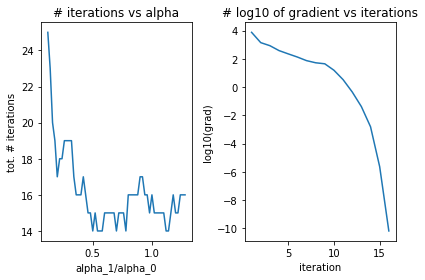}
\caption{Optimization on flag manifold $\Flag(30, 20, 10; 1000, \R)$ using a trust-region solver. Left, number of total iterations versus $\alpha=\alpha_1/\alpha_0$. Right, $\log_{10}$ of gradient for $\alpha=1$}
\label{fig:FlagOpt}
\end{figure}

For another test of the Riemannian optimization framework, we consider a nonlinear weighted PCA (principal component analysis) problem, which could be solved by optimizing over the positive-semidefinite matrix manifold. Given a symmetric matrix $A\in\Herm{\sfT}{\R}{n}$ and a weight vector $W\in\R^n$ we want to minimize the cost function
$$\Tr(A-YPY^{\sfT})\diag(W)(A-YPY^{\sfT})=\Tr W_d(A^2-AYPY^{\sfT}-YPY^{\sfT}A + YP^2Y^{\sfT}) $$
of a positive-semidefinite matrix $S = YPY^{\sfT}\in\Sp{\R}{p}{n}$, with $Y\in\St{\R}{p}{n}, P\in \Sd{\R}{p}$. Here, $W_d$ denotes the diagonal matrix $\diag(W)$. When $W$ has identical weight $\lambda$, $W_d = \lambda I_n$, expanding the cost function, we need to minimize $\Tr P^2 - 2\Tr Y^T A Y P$ in $Y$ and $P$, which implies $P = Y^T A Y$. Thus, the problem is optimizing $ -\Tr (Y^{\sfT}AY)^2$ over the Stiefel manifold (actually over the Grassmann manifold as the function is invariant when $Y$ is multiplied on the right by an orthogonal matrix), which could be considered as a quadratic PCA problem. When $W$ has non identical weights, it is difficult to reduce the problem to a Stiefel manifold, hence we optimize over the positive-semidefinite manifold, with a trust-region solver.

The cost function from $\cM = \St{\R}{p}{n}, \times \Sd{\R}{p}$ extends to $\cE=\R^{n\times p}\times \R^{p\times p}$, and is denoted by $\hatf(Y, P)$. For a horizontal tangent vector $\xi = (\xi_Y, \xi_P)$ at $(Y, P)\in\cM$
$$\begin{gathered}(\rD_{\xi}\hatf)(Y, P) = \Tr W_d(-A\xi_YPY^{\sfT} - AY\xi_PY^{\sfT} -AYP\xi_Y^{\sfT}\\
-\xi_YPY^{\sfT}A -Y\xi_PY^{\sfT}A -YP\xi_Y^{\sfT}A
+ \xi_YP^2Y^{\sfT} + Y(\xi_PP + P\xi_P)Y^{\sfT} + YP^2\xi_Y^{\sfT}) \end{gathered}$$
We have $\Tr W_d(-A\xi_YPY^{\sfT}  -AYP\xi_Y^{\sfT} -\xi_YPY^{\sfT}A -YP\xi_Y^{\sfT}A) =-2\Tr(AW_d + W_dA)YP\xi_Y^{\sfT}$, $\Tr W_d(\xi_YP^2Y^{\sfT} + YP^2\xi_Y^{\sfT}) =2\Tr W_dYP^2\xi_Y^{\sfT}$ and similar equalities for $\xi_P$ give us
$$\hgradf= (-4\sym{\sfT}(A W_d)YP + 2W_dYP^2, -2\sym{\sfT}(Y^{\sfT}W_d(AY - Y P)))$$
The ambient Hessian $\hhessf(\xi)$ follows from a directional derivative calculation
$$\begin{gathered}\hhessf(\xi) = -4\sym{\sfT}(AW_d)(\xi_YP + Y\xi_P) +
            2W_d\xi_YP^2 + 2W_dY(\xi_PP + P\xi_P),\\
            -2\sym{\sfT}(\xi_Y^{\sfT}W_d(AY - Y  P))
            - 2 \sym{\sfT}(Y^{\sfT}W_d(A\xi_Y - \xi_Y  P - Y \xi_P)))
\end{gathered}$$
The Riemannian gradient is computed as $\Pi_{\cH, \sfg}\sfg^{-1}\hgradf$, with $\Pi_{\cH, \sfg}$ is given by \cref{eq:psd_proj}, the Riemannian Hessian is computed from \cref{eq:rhess11}, with all components given in \cref{prop:psd_conn}. We use the built-in trust region solver in \cite{JMLR:v17:16-177}.

In our experiment (implemented in the notebook \href{https://github.com/dnguyend/ManNullRange/blob/master/colab/WeightedPCA.ipynb}{colab/WeightedPCA.ipynb} in \cite{Nguyen2020riemann}), we take $n= 1000, p = 50$, with $A$ and $W$ generated randomly. To find the optimum $S=YPY^{\sfT}$, we optimize with $\alpha_0 = \alpha_1 = 1$, with $\beta$ is $0.1$ for the first $20$ iterations, $\beta=10$ for the next $20$ and $\beta=30$ for the remaining iterations. This choice of $\beta$ comes from our limited experiments, we find varying $\beta$ has a strong effect on the speed of convergence, and updating $\beta$ as such gives better convergence rates than a static $\beta$. Philosophically, the small starting $\beta$ could be thought of as focusing first on aligning the subspace. The convergence graph is summarized in \cref{fig:WeightedPCA}. We hope to revisit the topic with a more systematic study in future works.
\begin{figure}
  \centering
\includegraphics[scale=0.4]{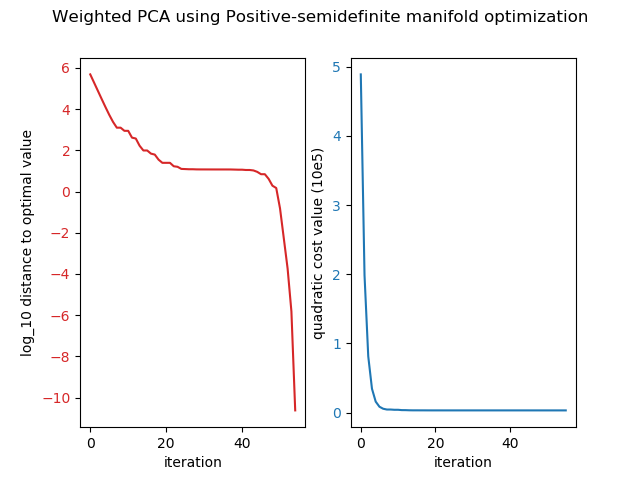}
\caption{Weighted PCA using Positive-semidefinite manifold optimization with a trust-region solver. Left, $\log_{10}$ of distance to the optimal value. Right, quadratic cost by iteration.}
\label{fig:WeightedPCA}
\end{figure}
\section{Conclusion} In this paper, we have proposed a framework to compute the Riemannian gradient, Levi-Civita connection, and the Riemannian Hessian effectively when the constraints, the symmetry of the problem, and the metrics are given analytically and have applied the framework to several manifolds important in applications. We look to apply the results in this paper to several problems in optimization and statistical learning. The optimization platform for positive-semidefinite matrices should help to learn sparse plus low-rank probability densities in statistical problems. We hope the research community will find the method useful in future works.
\bibliographystyle{amsplain}
\bibliography{RiemannianSymbolic}
\end{document}